\documentclass[journal]{IEEEtran}
\ifCLASSINFOpdf
\else
\fi
\usepackage[usenames,dvipsnames]{xcolor}
\usepackage{fancyhdr}
\usepackage{amsmath,amsfonts,amsbsy,amsgen,amscd,mathrsfs,amssymb}
\usepackage{amsthm}
\usepackage{url}
\usepackage{eurosym}
\usepackage{tikz}
\usetikzlibrary{matrix,arrows,shapes}
\usepackage{subfig}
\usepackage{microtype}
\usepackage{enumitem}
\usepackage[draft]{changes}
\definechangesauthor[color=blue]{ml}
\definechangesauthor[color=red]{da}

\definecolor{dark-gray}{gray}{0.3}
\definecolor{dkgray}{rgb}{.4,.4,.4}
\definecolor{dkblue}{rgb}{0,0,.5}
\definecolor{medblue}{rgb}{0,0,.75}
\definecolor{rust}{rgb}{0.5,0.1,0.1}

\usepackage{graphicx}
\usepackage{booktabs,longtable,tabu} 
\usepackage{multirow} 
\usepackage{float}
\usepackage[T1]{fontenc}

\usepackage{fourier}
\usepackage{charter}
\usepackage{bm} 

\graphicspath{{figures/}}

\numberwithin{equation}{section} 

\providecommand{\mathbold}[1]{\bm{#1}}  

\newtheorem{bigthm}{Theorem}

\newtheorem{theorem}{Theorem}[section]
\newtheorem{lemma}[theorem]{Lemma}

\newtheorem{proposition}[theorem]{Proposition}

\newtheorem{corollary}[theorem]{Corollary}

\theoremstyle{definition}

\newtheorem{example}[theorem]{Example}
\newtheorem{remark}[theorem]{Remark}

\renewcommand{\phi}{\varphi}

\newcommand{\e}{\varepsilon}

\renewcommand{\mid}{\mathrel{\mathop{:}}} 

\newcommand{\zerovct}{\vct{0}} 

\newcommand{\Id}{\mathbf{I}}
\newcommand{\onemtx}{\boldsymbol{1}}

\providecommand{\mathbbm}{\mathbb} 

\newcommand{\R}{\mathbbm{R}}

\newcommand{\polar}{\circ}

\newcommand{\diff}[1]{\mathrm{d}{#1}}

\newcommand{\minimize}{\text{minimize}\quad}
\newcommand{\subjto}{\quad\text{subject to}\quad}

\newcommand{\argmin}{\operatorname*{arg\; min}}

\newcommand{\Prob}{\mathbbm{P}}

\newcommand{\Expect}{\operatorname{\mathbb{E}}}

\newcommand{\diag}{\operatorname{diag}}

\newcommand{\vct}[1]{\mathbold{#1}}
\newcommand{\mtx}[1]{\mathbold{#1}}

\newcommand{\transp}[1]{#1^{T}}

\newcommand{\Proj}{\ensuremath{\mtx{\Pi}}} 

\newcommand{\ip}[2]{\langle {#1}, {#2} \rangle}

\newcommand{\norm}[1]{\Vert {#1}\Vert}

\newcommand{\lone}[1]{\norm{#1}_{1}}

\DeclareMathOperator{\dist}{dist}

\newcommand{\Desc}{\mathcal{D}}

\newcommand{\sdim}{\delta}

\newcommand{\IR}{\mathbbm{R}}
\newcommand{\veps}{\varepsilon}

\newcommand{\relint}{\operatorname{relint}}
\newcommand{\cone}{\operatorname{cone}}

\newcommand{\inter}{{\operatorname{int}}}
\newcommand{\conv}{\operatorname{conv}}

\newcommand{\D}{\mathcal{D}}
\renewcommand{\P}{\mathcal{P}}

\DeclareMathOperator{\ima}{im}

\newcommand{\Ren}{\mathcal{R}}

\newcommand{\resTMP}[2]{#1\to#2}

\newcommand{\nres}[3]{\|\vct{#1}\|_{\resTMP{#2}{#3}}}
\newcommand{\sres}[3]{\sigma_{\resTMP{#2}{#3}}(\vct{#1})}

\newcommand{\nrest}[3]{\|\vct{#1}^T\|_{\resTMP{#2}{#3}}}
\newcommand{\srest}[3]{\sigma_{\resTMP{#2}{#3}}(\vct{#1}^T)}

\newcommand{\srestm}[3]{\sigma_{\resTMP{#2}{#3}}(-\vct{#1}^T)}

\newcommand{\nresdag}[3]{\|\vct{#1}^\dagger\|_{\resTMP{#2}{#3}}}

\newcommand{\RCD}[3]{\Ren_{#2,#3}(\vct{#1})}

\newcommand{\DA}{\vct{\Delta A}}

\begin{document}
%
\title{Effective condition number bounds for convex regularization}
%
%
%

\author{Dennis Amelunxen,
        Martin Lotz
        and~Jake Walvin
\thanks{D. Amelunxen was with the Department
of Mathematics, City University of Hong Kong, Tat Chee Avenue, 
Kowloon Tong, Hong Kong}
\thanks{M. Lotz is with the Mathematics Institute, Zeeman Building, 
University of Warwick, Coventry CV4 7AL, U.~K.}
\thanks{J. Walvin was with the School of Mathematics, 
Alan Turing Building, Univeristy of Manchester, Manchester M13 9PL, U.~K.}
\thanks{Manuscript received May 17, 2018; revised September 27, 2019.}}

%
%

\markboth{IEEE Transactions on Information Theory,~Vol.~xx, No.~x, September~2019}%
{Amelunxen \MakeLowercase{\textit{et al.}}: Effective condition number bounds for convex regularization}
%



\maketitle

\begin{abstract}
We derive bounds relating Renegar's condition number to quantities that govern the statistical performance of convex regularization in settings that include the $\ell_1$-analysis setting. 
Using results from conic integral geometry, we show that the bounds can be made to depend only on a random projection, or restriction, of the analysis operator to a lower dimensional space, and
can still be effective if these operators are ill-conditioned. As an application, we get new bounds for the undersampling phase transition of composite convex regularizers. 
Key tools in the analysis are Slepian's inequality and the kinematic formula from integral geometry. 
\end{abstract}


%
\IEEEpeerreviewmaketitle

\section{Introduction}\label{sec:intro}
A well-established approach to solving linear inverse problems with missing information is by means of convex regularization.
In one of its manifestations, this approach amounts to solving 
the minimization problem
\begin{equation}\label{eq:conv-constr}
 \minimize f(\vct{x}) \ \subjto \norm{\mtx{\Omega}\vct{x}-\vct{b}}_2\leq \e,
\end{equation}
where $\mtx{\Omega}\in \R^{m\times n}$ represents an underdetermined linear operator and $f(\vct{x})$ is a suitable proper convex function, informed by the application at hand. The typical example is $f(\vct{x})=\lone{\vct{x}}$, known to promote sparsity, but many other functions have been considered in different settings.

While there are countless algorithms and heuristics to compute or approximate solutions of~\eqref{eq:conv-constr} and related problems, the more fundamental question is: when does a solution of~\eqref{eq:conv-constr} actually ``make sense''? The latter is important because one is usually not interested in a solution of~\eqref{eq:conv-constr} per se, but often uses this and related formulations as a proxy for a different, much more intractable problem. The best-known example is the use of the $1$-norm to obtain a sparse solution~\cite{FR:13}, but other popular settings are the total variation norm and its variants for signals with sparse gradient, or the nuclear norm of a matrix when aiming at a low-rank solution. 

Regularizers often take the form $f(\vct{x})=g(\mtx{D}\vct{x})$ for a linear map $\mtx{D}$, as in the cosparse recovery setting~\cite{elad2007analysis,candes2011compressed,NDEG:13}, where $f(\vct{x})=\norm{\mtx{D}\vct{x}}_1$ for an analysis operator $\mtx{D}\in \R^{p\times n}$ with possibly $p\geq n$. In this article we present general bounds relating the performance of~\eqref{eq:conv-constr} to properties of $g$ and the conditioning of $\mtx{D}$. Moreover, we show that for the analysis we can replace $\mtx{D}$ with a {\em random projection} applied to $\mtx{D}$, where the target dimension of this projection is independent of the ambient dimension $n$ and only depends on intrinsic properties of the regularizer $g$.

\subsection{Performance measures for convex regularization}
Various parameters have emerged in the study of the performance of problems such as~\eqref{eq:conv-constr}. Two of the most fundamental ones depend on the {\em descent cone} $\Desc(f,\vct{x}_0)$ of the function $f$ at $\vct{x}_0$, defined as the convex cone of all directions in which $f$ decreases. These parameters are
\begin{itemize}
\item the statistical dimension $\delta(f,\vct{x}_0):=\delta(\Desc(f,\vct{x}_0))$, or equivalently the squared Gaussian width, of the descent cone $\Desc(f,\vct{x}_0)$ of $f$ at a solution $\vct{x}_0$ (cone of direction from $\vct{x}_0$ in which $f$ decreases), which  determines the admissible amount of undersampling $m$ in~\eqref{eq:conv-constr} in the noiseless case ($\e=0$), in order to uniquely recover a solution $\vct{x}_0$\footnote{Strictly speaking, this is a result for {\em random} measurement matrices and holds with high probability.};
\item Renegar's condition number $\Ren_C(\mtx{\Omega})$ of $\mtx{\Omega}$ with respect to the descent cone $C=\Desc(f,\vct{x}_0)$ of $f$ at a point $\vct{x}_0$, which bounds the recovery error $\norm{\vct{x}-\vct{x}_0}_2$ of a solution $\vct{x}$ of~\eqref{eq:conv-constr}.
\end{itemize}

Before stating the results linking these two parameters, we briefly define them and outline their significance.
The statistical dimension of a convex cone is defined as the expected squared length of the projection of a Gaussian vector $\vct{g}$ onto a cone: $\delta(C)=\Expect[\norm{\Proj_C(\vct{g})}^2]$ (see Section~\ref{sec:stat-dim} for a principled derivation; unless otherwise stated, $\norm{\cdot}$ refers to the $2$-norm). It has featured as a proxy to the squared Gaussian width in~\cite{stojnic10,CRPW:12} and as the main parameter determining phase transitions in convex optimization~\cite{edge}. 
More precisely, let $\vct{x}_0 \in \R^n$, $\mtx{\Omega} \in \R^{m \times n}$ and $\vct{b} = \mtx{A}\vct{x}_0$.  Consider the optimization problem
\begin{equation}\label{eq:lin-inv-gen}
 \minimize f(\vct{x}) \subjto \mtx{\Omega}\vct{x}=\vct{b},
\end{equation}
which we deem to {\em succeed} if the solution coincides with $\vct{x}_0$.
In~\cite[Theorem II]{edge} it was shown that for any $\eta\in (0,1)$, when $\mtx{\Omega}$ has Gaussian entries, then
\begin{align*}
m &\geq \delta(f, \vct{x}_0) + a_{\eta} \sqrt{n}\\
&\Longrightarrow\quad
\text{\eqref{eq:lin-inv-gen} succeeds with probability~$\geq 1-\eta$}; \\
m &\leq \delta(f, \vct{x}_0) - a_{\eta} \sqrt{n}\\
&\Longrightarrow\quad
\text{\eqref{eq:lin-inv-gen} succeeds with probability~$\leq \eta$,}
\end{align*}
with $a_{\eta} := 4\sqrt{\log(4/\eta)}$. For $f(\vct{x})=\|\vct{x}\|_1$, the relative statistical dimension has been determined precisely by Stojnic~\cite{stojnic10}, and his results match previous derivations by Donoho and Tanner (see~\cite{dota:09a} and the references).
In addition, the statistical dimension / squared Gaussian width also features in the error analysis of the generalized LASSO problem~\cite{oymak2013squared}, as the minimax mean squared error (MSE) of proximal denoising~\cite{DJM:13, oymak2016sharp}, to study computational and statistical tradeoffs in regularization~\cite{chandrasekaran2013computational}, and in the context of structured regression~(\cite{han2017isotonic} and references). 

To define Renegar's condition number, first recall the classical condition number of a matrix $\vct A\in\IR^{m\times n}$, defined as the ratio of the 
operator norm and the smallest singular value. Using the notation
$\norm{\vct A} := \max_{\vct x\in S^{n-1}} \norm{\vct{Ax}}$, $\sigma(\vct A) := \min_{\vct x\in S^{n-1}} \norm{\vct{Ax}}$,
the classical condition number is given by
  \[ \kappa(\vct A) = \min\Bigg\{ \frac{\norm{\vct A}}{\sigma(\vct A)},\frac{\norm{\vct A}}{\sigma(\vct A^T)}\Bigg\} . \]
Renegar's condition number arises when replacing the source and target vector spaces $\R^{n}$ and $\R^{m}$ with convex cones.
Let $C\subseteq\IR^n$, $D\subseteq\IR^m$ be closed convex cones, and let $\vct A\in\IR^{m\times n}$. 
Define restricted versions of the 
norm and the singular value:
\begin{align}\label{eq:def-nres,sres}
   \nres{A}{C}{D} & := \max_{\vct x\in C\cap S^{n-1}} \|\Proj_D(\mtx{A}\vct{x})\| , \\
   \sres{A}{C}{D} & := \min_{\vct x\in C\cap S^{n-1}} \|\Proj_D(\mtx{A}\vct x)\| ,
\end{align}
where $\Proj_D\colon\IR^m\to D$ denotes the orthogonal projection, i.e., $\Proj_D(\vct y) = \argmin\{ \|\vct y-\vct z\|\mid \vct z\in D\}$.

Renegar's condition number is defined as
\begin{equation}\label{eq:renegar-def}
  \Ren_C(\mtx{A}) := \min\Bigg\{\frac{\|\vct A\|}{\sres{A}{C}{\R^m}},\frac{\|\vct A\|}{\srestm{A}{\R^m}{C}}\Bigg\}.
\end{equation}
In what follows, we simply write $\sigma_C(\mtx{A}):=\sres{A}{C}{\R^m}$ for the smallest cone-restricted singular value.
As mentioned before, Renegar's condition number features implicitly in error bounds solutions of~\eqref{eq:conv-constr}: if $\vct{x}_0$ is a feasible point and $\hat{\vct{x}}$ is a solution of~\eqref{eq:conv-constr}, then $\norm{\hat{\vct{x}}-\vct{x}_0}\leq 2 \e \Ren_{\Desc(f,\vct{x}_0)}(\mtx{\Omega})/\norm{\mtx{\Omega}}$ (see, for example,~\cite{CRPW:12}).
Renegar's condition number was originally introduced to study the complexity of linear programming~\cite{rene:95b}, see ~\cite{VRP:07} for an analysis of the running time of an interior-point method for the convex feasibility problem in terms of this condition number,
and~\cite{Condition} for a discussion and references. In~\cite{roulet2015computational}, Renegar's condition number is used to study restart schemes for algorithms such as NESTA~\cite{becker2011nesta} in the context of compressed sensing.

Unfortunately, computing or even estimating the statistical dimension or condition numbers is notoriously difficult for all but a few examples. For the popular case $f(\vct{x})=\norm{\vct{x}}_1$, an effective method of computing $\delta(f,\vct{x}_0)$ was developed by Stojnic~\cite{stojnic10}, and subsequently generalized in~\cite{CRPW:12}, see also~\cite[Recipe 4.1]{edge}. In many practical settings the regularizer $f$ has the form $f(\vct{x})=g(\mtx{D}\vct{x})$ for
a matrix $\mtx{D}$, such as in the cosparse or $\ell_1$-analysis setting where $f(\vct{x})=\norm{\mtx{D}\vct{x}}_1$.
Even when it is possible to accurately estimate the statistical dimension (and thus, the permissible undersampling) for a function $g$, the method may fail for a composite function $g(\mtx{D}\vct{x})$, due to a lack of certain separability properties~\cite{zhang2016precise} (see~\cite{genzel2017ell} for recent bounds in the $\ell_1$-analysis setting).

\subsection{Main results - deterministic bounds} In this article we derive a characterization of Renegar's condition number associated to a cone as a measure of how much the statistical dimension can change under a linear image of the cone. 
The first result linking the statistical dimension with Renegar's condition is Theorem~\ref{thm:main-cond-bound}.
When using the usual matrix condition number, the upper bound in Equation~\eqref{eq:mu_r(TC)<=...-kappa-1} features implicitly in~\cite{kabanava2015robust,kabanava2015analysis} and appears to be folklore.

\begin{bigthm}\label{thm:main-cond-bound}
Let $C\subseteq\IR^n$ be a closed convex cone, and
$\delta(C)$ the statistical dimension of $C$. Then for $\vct A\in\IR^{p\times n}$,
\begin{equation}\label{eq:mu_r(TC)<=...-1}
  \delta(\mtx{A}C) \leq \Ren_C(\mtx{A})^2\cdot \delta(C),
\end{equation}
where $\Ren_{C}(\mtx{A})$ is Renegar's condition number associated to the matrix $\mtx{A}$ and the cone $C$.
If $p\geq n$, $\mtx{A}$ has full rank, and $\kappa(\mtx{A})$ denotes the matrix condition number of $\mtx{A}$, then
\begin{equation}\label{eq:mu_r(TC)<=...-kappa-1}
  \frac{\sdim(C)}{\kappa(\vct A)^2} \leq \sdim(\vct AC) \leq \kappa(\vct A)^2\,\cdot \sdim(C).
\end{equation}
\end{bigthm}

\begin{example}\label{ex:tv}
Consider the $n\times n$ finite difference matrix
\begin{equation*}
  \mtx{D} = \begin{pmatrix} 
             -1 & 1 & 0 & \cdots & 0 \\
             0 & -1 & 1 & \cdots & 0 \\
             0 & 0 & -1 & \cdots & 0 \\
             \vdots & \vdots & \vdots & \ddots & \vdots \\
             0 & 0 & 0 & \cdots & -1 
            \end{pmatrix}.
\end{equation*}
This matrix is usually defined with an additional column $(0,\dots,0,1)^T$, but for simplicity, and to work with a square matrix of full rank, we work with this truncated version.
The condition number is known to be of order $\Omega(n)$, making condition bounds using the normal matrix condition number useless.
Using Renegar's condition number with respect to a cone, on the other hand, can improve the situation dramatically. Consider, for example, the cone
\begin{equation*}
  C = \{ \vct{x}\in \R^n \mid x_1\geq 0, \ x_ix_{i+1}\leq 0 \text{ for } 1\leq i<n\}.
\end{equation*}
This cone is the orthant consisting of vectors with alternating signs. The cone-restricted singular value of $\mtx{D}$ is given by
\begin{align*}
  \sigma_C(\mtx{D})^2 &= \min_{\vct{x}\in C\cap S^{n-1}} \|\mtx{D}\vct{x}\|^2\\
  &= \min_{\vct{x}\in C\cap S^{n-1}}\sum_{i=1}^{n-1} (x_{i+1}-x_i)^2 + x_n^2\\
  &= \min_{\vct{x}\in C\cap S^{n-1}} 2-x_1^2-\sum_{i=1}^{n-1} 2x_ix_{i+1}\geq 1.
\end{align*}
Using the same expression for $\|\mtx{D}\vct{x}\|^2$, we see that the square of the operator norm is bounded by $4$, so that the square of Renegar's condition number with respect to this cone is bounded by $4$. If, on the other hand, $C$ is the non-negative orthant, then Renegar's condition number coincides with the normal matrix condition number. Intuitively, Renegar's condition number gives an improvement if the cone $C$ captures a portion of the ellipsoid defined by $\mtx{D}\mtx{D}^T$ that is not too eccentric. Other examples when Renegar's condition number gives significant improvements is for small cones (such as the cone of increasing sequences) or cones contained in linear subspaces of small dimension (such as subdifferential cones of the $1$ or $\infty$ norms).
\end{example}

Theorem~\ref{thm:main-cond-bound} translates into a bound for the statistical dimension of convex regularizers by observing that if $f(\vct{x}) = g(\mtx{Dx})$ with invertible $\mtx{D}$, then (see Section~\ref{sec:applications}) the descent cone of $f$ at $\vct{x}_0$ is given by $\Desc(f,\vct{x}_0) = \mtx{D}^{-1}\Desc(g,\mtx{D}\vct{x}_0)$. Throughout this paper, we will use $\mtx{A}$ for the transformation matrix in the setting of convex cones, and $\mtx{D}$ for the matrix appearing in a regularizer.

\begin{corollary}\label{cor:fromA}
Let $f(\vct{x}) = g(\vct{D}\vct{x})$, where $g$ is a proper convex function and let $\mtx{D}\in \R^{n\times n}$ be non-singular. Then
\begin{equation*}
  \delta(f,\vct{x}_0) \leq \Ren_{\Desc(g,\mtx{D}\vct{x}_0)}\left(\mtx{D}^{-1}\right) \cdot \delta(g,\mtx{D}\vct{x}_0).
\end{equation*}
In particular, 
\begin{equation*}
  \frac{\delta(g,\mtx{D}\vct{x}_0)}{\kappa(\mtx{D})^2}\leq \delta(f,\vct{x}_0) \leq \kappa(\mtx{D})^2 \cdot \delta(g,\mtx{D}\vct{x}_0).
\end{equation*}
\end{corollary}

\begin{remark}
It is interesting to compare the bounds in Corollary~\ref{cor:fromA} to the condition number bounds for sparse recovery by $\ell_1$-minimization from~\cite{kueng2014ripless}. If $\mtx{D}\in \R^{n\times n}$ is invertible, then Problem~\ref{eq:lin-inv-gen} with $f(\vct{x})=g(\mtx{D}\vct{x})$ is mathematically equivalent to
\begin{equation}\label{eq:lin-inv-new}
 \minimize g(\vct{y}) \subjto \mtx{\Omega}\mtx{D}^{-1}\vct{y}=\vct{b}.
\end{equation}
In~\cite{kueng2014ripless}, the authors consider measurement matrices $\mtx{\Omega}$ for which the rows $\vct{\omega}^T$ are sampled according to a distribution with covariance $\Expect[\vct{\omega}\vct{\omega}^T]$.
In the isotropic case where the covariance is a multiple of the identity matrix, the measurement ensemble in~\ref{eq:lin-inv-new} is non-isotropic and the covariance matrix has condition number proportional to $\kappa(\mtx{D})^2$. In~\cite[Theorem 2]{kueng2014ripless}, a lower bound on the number of measurements needed for recovering a signal is given that involves the condition number of the covariance matrix. The bounds in~\cite{kueng2014ripless} apply directly to the number of measurements for recovery by $\ell_1$-minimization, and under rather general assumptions on the distribution. Moreover, the bounds in~\cite{kueng2014ripless} rely on the condition number restricted to sparse vectors, while in our case we consider Renegar's condition number with respect to the descent cone.
The bounds in Corollary~\ref{cor:fromA} also apply to any convex regularizer, and their applicability to sparse recovery is via the proxy of the statistical dimension, and thus restricted to situations in which this parameter delivers recovery bounds.
\end{remark}

While Renegar's condition number, defined by restricting the smallest singular value to a cone, can improve the bound, computing this condition number is not always practical. Using polarity~\eqref{eq:complement}, we get the following version of the bound that ensures that the right-hand side is always bounded by $n$. 

\begin{corollary}\label{prop:improved-cond}
Let $C\subseteq\IR^n$ be a closed convex cone, and
$\delta(C)$ the statistical dimension of $C$. Let $\vct A\in\IR^{n\times n}$ be non-singular. Then
\begin{equation*}
  \sdim(\vct AC) \leq \kappa(\vct A)^{-2} \cdot \sdim(C) + \left(1-\kappa(\vct A)^{-2}\right)\cdot n.
\end{equation*}
If $f(\vct{x}) = g(\vct{D}\vct{x})$, where $g$ is a proper convex function and $\mtx{D}\in \R^{p\times n}$ with $p\geq n$, then
\begin{equation}\label{eq:convreg-cond-bound}
   \delta(f,\vct{x}_0) \leq \kappa(\mtx{D})^{-2} \cdot \delta(g,\mtx{D}\vct{x}_0) + \left(1-\kappa(\mtx{D})^{-2}\right) \cdot n.
\end{equation}
\end{corollary}

The simple proof of Corollary~\ref{prop:improved-cond} is given in Section~\ref{sec:improved}.
One can interpret the upper bounds in Corollary~\ref{prop:improved-cond} as interpolating between the statistical dimension of $C$ and the ambient dimension $n$. 

\begin{remark}
The restriction to invertible dictionaries $\mtx{D}$ may look limiting at first, but a closer look reveals that it is not necessary when working with the subdifferential cone instead of the descent cone (see Section~\ref{sec:conv-anal} for the relevant definitions and background). In fact, given a proper convex function $f(\vct{x})=g(\mtx{D}\vct{x})$, the statistical dimensions of the descent cone and that of the subdifferential cone are related as
\begin{equation*}
  \delta(f,\vct{x}_0) = n-\delta(\cone(\partial f(\vct{x}_0))). 
\end{equation*}
Therefore, lower bounds on the statistical dimension of the subdifferential cone imply upper bounds on the statistical dimension of $f$. It is well known that $\cone(\partial f(\vct{x}_0))=\mtx{D}^T\cone(\partial g(\mtx{D}\vct{x}_0))$, and therefore if $\mtx{D}\in \R^{p\times n}$ with $p\leq n$, we can apply the lower bound from~\eqref{eq:mu_r(TC)<=...-kappa-1}. In applications, however, the case $p\geq n$ is of interest. In this case one should note that the subdifferential cone is often contained in a linear subspace of dimension at most $n$, and by common invariance properties of the statistical dimension (Section~\ref{sec:stat-dim}) it is enough to work with the restriction of $\mtx{D}^T$ to this lower dimensional subspace. Proposition~\ref{prop:l1analysis} illustrates this idea in the case of the $1$-norm.
\end{remark}

In the statement of the proposition below, we use the notation $\mtx{A}_I$ for the submatrix of a matrix $\mtx{A}$ with columns indexed by $I\subset [n]=\{1,\dots,n\}$, and denote by $I^c=[n]\backslash I$ the complement of $I$. The proof is postponed to Section~\ref{sec:applications}. 

\begin{proposition}\label{prop:l1analysis}
Let $\mtx{D}\in \R^{p\times n}$, $p\geq n$, be such that all $n\times n$ minors of $\mtx{D}$ have full rank,
and $\mtx{A}\in \R^{m\times n}$ with $m\leq n$. Consider the problem
\begin{equation}\label{eq:cosparse-conv}
  \minimize \norm{\mtx{D}\vct{x}}_1 \quad \subjto \mtx{\Omega}\vct{x}=\vct{b}.
\end{equation}
Let $\vct{x}_0$ be such that $\mtx{\Omega}\vct{x}_0=\vct{b}$, and such that $\vct{y}_0=\mtx{D}\vct{x}_0$ is $s$-sparse with support $I\subset [p]$. Let $\mtx{C}\in \R^{n\times p-s+1}$ be a matrix whose first $p-s$ columns consist of the columns of $\mtx{D}^T$ that are indexed by $I^{c}$, and the last column is $\vct{c}_{p-s+1}=\frac{1}{\sqrt{s}}\sum_{j\in I}\mathrm{sign}((\vct{y}_0)_j)\vct{d}_j$, where the vectors $\vct{d}_j$ denote the columns of $\mtx{D}^T$. 
Then
\begin{equation}\label{eq:upper-bound-1}
  \delta(\norm{\mtx{D}\cdot}_1,\vct{x}_0) \leq \kappa(\mtx{C})^{-2} \cdot \delta(\norm{\cdot}_1,\mtx{D}\vct{x}_0) + \left(1-(p/n)\kappa(\mtx{C})^{-2}\right) \cdot n
\end{equation}
In particular, given $\eta\in (0,1)$, Problem~\eqref{eq:cosparse-conv} with Gaussian measurement matrix succeeds with probability $1-\eta$ if
\begin{equation*}
  m\geq \kappa(\vct{C})^{-2}\cdot \delta(\norm{\cdot}_1,\mtx{D}\vct{x}_0)+\left(1-(p/n)\kappa(\mtx{C})^{-2}\right)\cdot n+a_{\eta} \sqrt{n}, 
\end{equation*}
\end{proposition}

\begin{example}
An illustrative example is the finite difference matrix $\mtx{D}$ of example~\ref{ex:tv}.
The regularizer $f(\vct{x})=\lone{\mtx{Dx}}$ is a one-dimensional version of a total variation regularizer, and is used to promote gradient sparsity. The standard method~\cite[Recipe 4.1]{edge} for computing the statistical dimension of the descent cone of $f$ is not easily applicable here, as this regularizer is not separable~\cite{zhang2016precise} (in fact, it would require a careful analysis of the structure of the signal with sparse gradient to be recovered). 
The standard condition number bound Theorem~\ref{thm:main-cond-bound} is also not applicable,
as it is known that the condition number satisfies $\kappa(\mtx{D})\geq \frac{2(n+1)}{\pi}$. Figure~\ref{fig:tvplot} plots the upper bound of Proposition~\ref{prop:l1analysis} for signals with random support location and sparsity ranging from $1$ to $200$, and compares it to the actual statistical dimension computed by Monte Carlo simulation. As can be seen in this example, the upper bound is not very useful because of the large condition numbers involved.
\begin{figure}[h!]
\begin{center}
\includegraphics[width=0.5\textwidth]{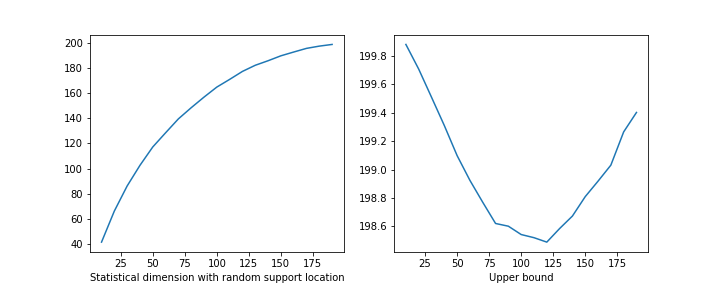}
\end{center}
\caption{The statistical dimension of $\norm{\vct{D}\cdot}_1$ for different sparsity levels and the upper bound~\eqref{eq:upper-bound-1}.}\label{fig:tvplot}
\end{figure}
\end{example}

\begin{remark}
It is natural to ask for which dictionaries $\mtx{D}$ Proposition~\ref{prop:l1analysis} gives good bounds. 
This clearly also depends on the support of the signal one wishes to recover. A closer look at the matrix $\mtx{C}$ in the case of the finite difference matrix and for {\em monotonely increasing} signals shows that $\mtx{C}$ is (up to rows of zeros) itself a finite difference matrix of order $n-s+1$, and the quality of the bounds increases with the size of the support. Another natural example is when $\mtx{D}\in \R^{p\times n}$ is a Gaussian random matrix (that is, a matrix whose entries are independent standard normal distributed random variables). In this case, the invariance properties of Gaussians imply that the matrix $\mtx{C}$ is again a Gaussian matrix in $\R^{n\times p-s+1}$. For such matrices, the condition number is known to be of order $(\sqrt{n}+\sqrt{p-s+1})/(\sqrt{n}-\sqrt{p-s+1})$ with high probability, see for example~\cite[Theorem 5.32]{Vershynin2012}. In this example we see again that if the support is large, $s\approx p$, then the condition number is close to $1$ and the bound becomes useful. 
\end{remark}

Note that so far we have seen two types of bounds: those based on the upper bound using Renegar's condition number in Theorem~\ref{thm:main-cond-bound}, which improve on the standard condition number by using the cone-restricted smallest singular value, and those based on duality and the lower bound of Theorem~\ref{thm:main-cond-bound}. The latter only work using the standard matrix condition number, but apply to the matrix restricted to the subspace generated by the cone of interest. Both bounds could yield good results for tight frames / well-conditioned matrices, but fail to give useful bounds in cases such as the finite difference matrix, or for redundant dictionaries $\mtx{D}$ for which the statistical dimension~$\delta(\|\cdot\|_1,\mtx{D}\vct{x}_0)$ is proportional to the (larger) ambient dimension. In the next section we discuss randomized improvements.

\subsection{Main results - probabilistic bounds} 
While Corollary~\ref{prop:improved-cond} ensures that the upper bound does not become completely trivial, when $\mtx{D}$ is ill-conditioned it still does not give satisfactory results, as seen in Example~\ref{ex:tv}. The second part, and main contribution, of our work is
an improvement of the condition bounds using randomization: using methods from conic integral geometry, we derive a ``preconditioned'' version of Theorem~\ref{thm:main-cond-bound}. The idea is based on the philosophy that a randomly oriented convex cone $C$ ought to behave roughly like a linear subspace of dimension $\delta(C)$. In that sense, the statistical dimension of a cone $C$ should be approximately invariant under projecting $C$ to a subspace of dimension close to $\delta(C)$. In fact, in Section~\ref{sec:tqc} we will see that for $n\geq m\gtrapprox \delta(C)$, we have
\begin{equation*}
  \Expect_{\mtx{Q}}\left[\delta(\mtx{P}_m\mtx{Q}C)\right] \approx \delta(C),
\end{equation*}
where $\mtx{P}_m$ is the projection on the the first $m$ coordinates and where the expectation is with respect to a random orthogonal matrix $\mtx{Q}$, distributed according to the normalized Haar measure on the orthogonal group.
From this it follows that the condition bounds should ideally depend not on the conditioning of $\mtx{D}$ itself, but on a generic projection of $\mtx{D}$ to linear subspace of dimension of order $\delta(C)$. 
For $m\leq n$ define
\begin{equation*}
  \overline{\kappa}_{m}^2(\mtx{A}) := \Expect_{\mtx{Q}}[\kappa(\mtx{P}_m\mtx{QA})^2], \quad \overline{\mathcal{R}}^2_{C,m}(\mtx{A}):=\Expect_{\mtx{Q}}\left[\Ren_C(\mtx{P}_m\mtx{QA})^2\right].
\end{equation*}

\begin{bigthm}\label{thm:B}
Let $C\subseteq \R^n$ be a closed convex cone and $\mtx{A}\in \R^{p\times n}$ be a matrix of full rank. Let $\eta\in (0,1)$ and assume that $m\geq \delta(C)+2\sqrt{\log(2/\eta)m}$. Then
\begin{equation*}
  \delta(\mtx{A}C)\leq \overline{\mathcal{R}}^2_{C,m}(\mtx{A}) \cdot \delta(C)+(n-m)\eta.
\end{equation*}
For the matrix condition number,
\begin{equation}\label{eq:convreg-upper-bound}
  \delta(\mtx{A}C) \leq \overline{\kappa}_{m}^2(\mtx{A}) \cdot \delta(C)+(n-m)\eta.
\end{equation}
\end{bigthm}

As a consequence of Theorem~\ref{thm:B} we get the following preconditioned version of the previous bounds.

\begin{corollary}\label{cor:fromB}
Let $f(\vct{x}) = g(\vct{D}\vct{x})$, where $g$ is a proper convex function and $\mtx{D}\in \R^{n\times n}$ is non-singular. Let $\eta\in (0,1)$ and assume that $m\geq \delta(g,\mtx{D}\vct{x}_0)+2\sqrt{\log(2/\eta)m}$. Then
\begin{equation}\label{eq:convreg-upper-bound}
  \delta(f,\vct{x}_0) \leq 
  \overline{\mathcal{R}}^2_{\Desc(g,\vct{Dx}_0),m}(\mtx{D}^{-1})
   \cdot \delta(g,\vct{D}\vct{x}_0)+(n-m)\eta
\end{equation}
and 
\begin{equation*}
 \delta(f,\vct{x}_0) \leq \overline{\kappa}_m^2(\mtx{D}^{-1}) \cdot \delta(g,\mtx{D}\vct{x}_0)+(n-m)\eta.
\end{equation*}
\end{corollary}

\begin{example}
Consider a diagonal matrix $\mtx{\Sigma}$ and the average condition $\overline{\kappa}_m^2(\mtx{\Sigma})$. Intuitively, the average condition measures the expected eccentricity of the projection of an ellipsoid to a random subspace. 
\end{example}

\begin{example}\label{ex:D}
Using the finite difference matrix $\mtx{D}$ from Example~\ref{ex:tv}, note that it is physically not possible, nor do we aim to, locate the precise phase transition for the recovery with $f(\vct{x})=\lone{\mtx{D}\vct{x}}$ in terms of that of the $1$-norm, since the statistical dimension $\delta(f,\vct{x}_0)$ does not only depend on the sparsity pattern of $\mtx{D}\vct{x}_0$, but also on the location of the support. 
\end{example}

\subsection{Scope and limits of reduction}\label{sec:intro-universality}
The condition bounds in Theorem~\ref{thm:B} naturally lead to the question of how to compute or bound the condition number of a random projection of a matrix,
\begin{equation*}
 \kappa(\mtx{P}_m\mtx{QA}) \quad \text{ or } \quad \mathcal{R}_C(\mtx{P}_m\mtx{Q}\mtx{A})
\end{equation*}
where $\mtx{Q}\in O(n)$ is a random orthogonal matrix.
If $m=\lfloor\rho n\rfloor$ with $\rho\in (0,1)$, then in some cases the condition number $\kappa(\mtx{P}_m\mtx{QA})$ remains bounded with high probability as $n\to \infty$. Below we sketch how such condition numbers can be bounded. 

In what follows, let $\mtx{A}\in \R^{n\times n}$ be fixed and non-singular, and we write $\mtx{Q}_m=\mtx{P}_m\mtx{Q}$ for a random matrix with orthogonal rows, uniformly distributed on the Stiefel manifold.
We first reduce to the case of Gaussian matrices, for which tools are readily available. 
If $\mtx{G}\sim N(\zerovct,\onemtx)$ is an $m\times n$ random matrix with Gaussian entries, then $\mtx{Q}_m=(\mtx{G}\mtx{G}^T)^{-1/2}\mtx{G}$ is uniformly distributed on the Stiefel manifold, so that $\mathcal{R}_C(\mtx{Q}_m\mtx{A})$ has the same distribution as $\mathcal{R}_C((\mtx{G}\mtx{G}^T)^{-1/2}\mtx{G}\mtx{A})$. Using Lemma~\ref{le:product_bound}, we can bound (with probability one)
\begin{equation*}
  \mathcal{R}_C\left((\mtx{G}\mtx{G}^T)^{-1/2}\mtx{G}\mtx{A}\right) \leq \kappa\left((\mtx{G}\mtx{G}^T)^{-1/2}\right) \mathcal{R}_C\left(\mtx{G}\mtx{A}\right) =  \kappa(\mtx{G}) \mathcal{R}_C\left(\mtx{G}\mtx{A}\right),
\end{equation*}
transforming the problem into one in which the orthogonal matrix is replaced with a Gaussian one. The are different ways to estimate such condition numbers, the approach taken here is based on Gordon's inequality. We restrict the analysis to the classical matrix condition number, a more refined analysis using Renegar's condition number is likely to incorporate the Gaussian width of the cone. Moreover, using the invariance of the condition number under transposition, we consider $\kappa(\mtx{A}\mtx{G})$ with a $n\times m$ matrix $\mtx{G}$, $m\leq n$. 
An alternative, suggested by Armin Eftekhari, would be to appeal to the Hanson-Wright inequality~\cite{rudelson2013hanson,armin}, or more directly, the Bernstein inequality. 

\begin{proposition}
Let $\mtx{A}\in \R^{n\times n}$ and $\mtx{G}\in \R^{n\times m}$, with $m\leq n$. Then
\begin{equation}
  \Expect[\kappa(\mtx{\mtx{AG}})] \leq \frac{\norm{\mtx{A}}_F+\sqrt{m}\norm{\mtx{A}}_2}{\norm{\vct{A}}_F-\sqrt{m}\norm{\vct{A}}_2}
\end{equation}
whenever $\|\mtx{A}\|_F\geq \sqrt{m}\|\mtx{A}\|_2$.
\end{proposition}

Using a standard procedure one can show that the singular value and the norm will stay close to their expected values with high probability.
More specifically, one can use the above proposition as a basis for a weak average-case analysis of Renegar's condition number for random matrices of the form $\mtx{AG}$, as in~\cite{amelunxen2017average}.

\begin{proof}
We will derive the inequalities
\begin{equation*}
 \norm{\mtx{A}}_F -\sqrt{m}\norm{\mtx{A}}_2 \leq \Expect[\sigma(\mtx{AG})] \leq \Expect[\norm{\mtx{AG}}_2] \leq \norm{\mtx{A}}_F +\sqrt{m}\norm{\mtx{A}}_2.
\end{equation*}
where $\sigma$ denotes the smallest singular value. We will restrict to showing the lower bound, the upper bound follows similarly by using Slepian's inequality. Without lack of generality assume $\mtx{A}=\mtx{\Sigma}$ is diagonal, with entries $\sigma_1\geq \dots\geq \sigma_n$ on the diagonal, and assume $\sigma_1=1$. Define the Gaussian processes
\begin{equation*}
  X_{\vct{x},\vct{y}} = \ip{\mtx{G}\mtx{x}}{\mtx{\Sigma}\vct{y}}, \quad Y_{\vct{x},\vct{y}} = \ip{\vct{g}}{\vct{x}}+\ip{\vct{h}}{\mtx{\Sigma}\vct{y}},
\end{equation*}  
indexed by $\vct{x}\in S^{m-1}$, $\vct{y}\in S^{n-1}$, with $\mtx{g}\in \R^m$ and $\vct{h}\in \R^n$ Gaussian vectors. We get
\begin{align*}
  \Expect[(X_{\vct{x},\vct{y}} - X_{\vct{x}',\vct{y}'})^2] &= \norm{\vct{\Sigma y}}^2+\norm{\vct{\Sigma y}'}^2 -2\ip{\vct{x}}{\vct{x}'}\ip{\vct{\Sigma y}}{\vct{\Sigma y}'},\\ 
  \Expect[(Y_{\vct{x},\vct{y}} - Y_{\vct{x}',\vct{y}'})^2] &= \norm{\vct{\Sigma y}}^2+\norm{\vct{\Sigma y}'}^2 +2-2\ip{\vct{x}}{\vct{x}'}-2\ip{\mtx{\Sigma}\vct{y}}{\mtx{\Sigma}\vct{y}'},
\end{align*}
so that
\begin{align*}
  \Expect[(Y_{\vct{x},\vct{y}} - Y_{\vct{x}',\vct{y}'})^2] & - \Expect[(X_{\vct{x},\vct{y}} - X_{\vct{x}',\vct{y}'})^2] \\
  & = 2(1-\ip{\vct{x}}{\vct{x}'})(1-\ip{\vct{\Sigma y}}{\vct{\Sigma y}'})\\
  & \geq 0.
\end{align*}
This expression is $0$ if $\vct{x}=\vct{x}'$, and non-negative otherwise, since by assumption $\mtx{\Sigma}$ has largest entry equal to $1$. We can therefore apply Gordon's Theorem (see Section~\cite[9.2]{FR:13} or \cite[Theorem B.1]{amelunxen2014gordon}) to infer an inequality
\begin{align*}
  \Expect[\sigma(\mtx{\Sigma G})] & = \Expect[\min_{\vct{x}\in S^{m-1}}\max_{\vct{y}\in S^{n-1}} \ip{\vct{G x}}{\vct{\Sigma y}}]\\
   &= \Expect[\min_{\vct{x}\in S^{m-1}}\max_{\vct{y}\in S^{n-1}} X_{\vct{x},\vct{y}}]\\
  &\geq \Expect[\min_{\vct{x}\in S^{m-1}}\max_{\vct{y}\in S^{n-1}} Y_{\vct{x},\vct{y}}]\\
  & = \norm{\mtx{\Sigma}}_F - \sqrt{m}. 
\end{align*}
In general, if $\sigma_1\neq 1$, we replace $\mtx{\Sigma}$ by $\mtx{\Sigma}/\norm{\mtx{\Sigma}}_2=\mtx{\Sigma}/\norm{\mtx{A}}_2$, and obtain the desired bound.
\end{proof}

It would be interesting to characterize those matrices $\mtx{A}$ for which $\kappa(\mtx{P}_m\mtx{Q}\mtx{A})\approx 1$ using a kind of restricted isometry property, as for example in~\cite{oymak2015isometric}.
We leave a detailed discussion of the probability distribution of $\kappa(\mtx{P}_m\mtx{QA})$ and its ramifications for another occasion, and instead consider a special case.

\begin{example}
Consider again the matrix $\mtx{D}$ from Example~\ref{ex:tv}. For $\rho \in \{0.1,0.2,0.3,0.4\}$ and $n$ ranging from $1$ to $400$, $m=\lfloor \rho n\rfloor$, we plot the average condition number $\kappa(\mtx{DG})$, where $\mtx{G}\in \R^{n\times m}$ is a Gaussian random matrix. As $n$ increases, this condition number appears to converge to a constant value. We also plot the condition number $\kappa(\mtx{D}^{-1}\mtx{G})$, where $\mtx{D}^{-1}$ is the upper triangular matrix with non-zero entries $-1$. The different decay of the singular values leads to condition number that increase with $n$.

\begin{figure}[h!]
\centering
\includegraphics[width=0.45\textwidth]{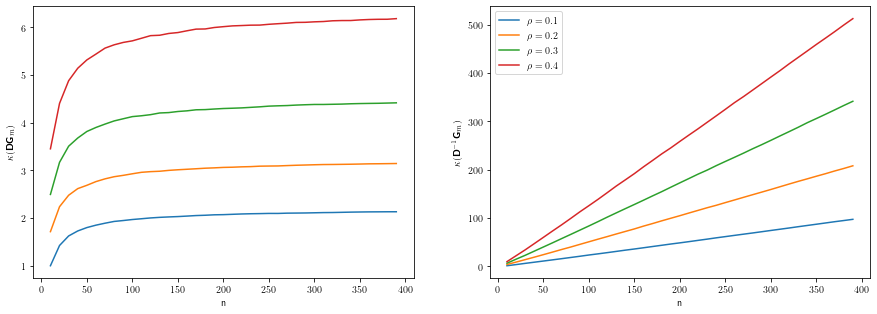}
\caption{Condition number $\kappa(\mtx{G}_m\mtx{D})$ for the matrix $\mtx{D}$ from Example~\ref{ex:tv}, and for its inverse. $\mtx{G}_m$ is the projection to the first $m=\lfloor \rho n\rfloor$ coordinates of a Gaussian $n\times n$ matrix $\mtx{G}$}
\end{figure}

As we saw in Example~\ref{ex:tv}, the operator norm of $\mtx{D}$ is bounded by $\norm{\vct{\sigma}}_\infty \leq 2$. 
The Frobenius norm, on the other hand, is easily seen to be $\|\mtx{D}\|_F=\|\sigma\|_2 = \sqrt{2n-1}$. 
Setting $m=\rho n$, the condition number thus concentrates on a value bounded by
\begin{equation*}
  \frac{\sqrt{2n-1}+2\sqrt{m}}{\sqrt{2n-1}-2\sqrt{m}} \approx \frac{1+\sqrt{2\rho}}{1-\sqrt{2\rho}},
\end{equation*}
which is sensible if $\rho<1/2$. We remark that, by construction, the bounds are not sharp, and also do not apply to the inverse $\mtx{D}^{-1}$.
\end{example}

\subsubsection{A note on applicability}
The previous discussion has shown that the condition number bounds need only consider the restricted condition number of a random projection of a matrix, rather than the full matrix condition. However, as the bounds are multiplicative, even small values (for example, $2$) lead to bounds for the the statistical dimension of the transformed cone that may not be practical. In addition, the statistical dimension of the reference cone also determines how small the projected dimension $m$ is allowed to become, further limiting the amount of potential reduction in condition. If, for example, $C$ is the descent cone of the $\ell_1$-norm, then the resulting bounds can only be used for the descent cones of the $\ell_1$-norm at very sparse vectors. The same applies when considering, instead of the difference matrix $\mtx{D}$ and its inverse, diagonal matrices with various forms of decay in the entries (this corresponds to a version of weighted $\ell_1$ recovery). In these cases, the expected condition of the randomly projected matrices can be improve dramatically, but still not enough to give non-trivial bounds across all sparsity levels.
This limitation is inherent to the notion of condition number:  Condition bounds are, by definition, pessimistic. In numerical analysis, they measure the worst case sensitivity of a problem to perturbations in the input. As such, it would be unrealistic to expect condition bounds to be able to accurately locate the statistical dimension of the descent cone of a composite regularizer, unless the matrix $\mtx{D}$ involved is close to orthogonal. 

\subsubsection{A note on distributions}
The results presented are based on integral geometry, and as such depend crucially on $\mtx{Q}$ being uniformly distributed in the orthogonal group with the Haar measure. By known universality results~\cite{oymak2015universality}, the results are likely to carry over to other distributions. In the context of this paper, however, we are neither interested in actually preconditioning the matrices involved, nor are we using them as a model for observation or measurement matrices as is common in compressive sensing. The randomization here is merely a technical tool to improve bounds based on the condition number, and the question of whether this is a ``realistic'' distribution is of no concern.

\subsection{Organisation of the paper}
In Section~\ref{sec:con-res-lin-op} we introduce the setting of conically restricted linear operators, the biconic feasibility problem, and Renegar's condition number in some detail. The characterization of this condition number in the generality presented here is new and of independent interest. Section~\ref{sec:lin_imag} derives the main condition bound. In Section~\ref{sec:conic-integral-geometry} we change the scene and give a brief overview of conic integral geometry, culminating in a proof of Theorem~\ref{thm:B} in Section~\ref{sec:improved}. Finally, in Section~\ref{sec:applications} we translate the results to the setting of convex regularizers. Appendix A presents some more details on the biconic feasibility problem, while Appendix B presents a general version of Gordon's inequality. While this version is more general than what is needed in this paper, it may be of independent interest.

\section{Conically restricted linear operators}\label{sec:con-res-lin-op}
In this section we discuss the restriction of a linear operator to closed convex cones and discuss Renegar's condition number in some detail.

\subsection{Restricted norm and restricted singular value}\label{sec:renegar}
Before discussing conically restricted operators in more detail, we record the following simple but useful lemma, which generalizes the relation $\ker\vct A=(\ima \vct A^T)^\bot$.

\begin{lemma}\label{lem:A^(-1)(D^*)}
Let $D\subseteq\IR^m$ be a closed convex cone. Then the polar cone is the inverse image of the origin under the projection map, $D^\polar := \{\vct z\in\IR^m\mid \langle \vct y,\vct z\rangle \leq 0 \text{ for all } \vct y\in D\} = \Proj_D^{-1}(\vct0)$. Furthermore, if $\vct A\in\IR^{m\times n}$, then
\begin{equation}\label{eq:A^(-1)(D^*)}
  \vct A^{-1}(D^\polar) = \big( \vct A^T D\big)^\polar ,
\end{equation}
where $\vct A^{-1}(D^\polar)=\{\vct x\in\IR^n\mid \vct{Ax}\in D^\polar\}$ denotes the inverse image of~$D^\polar$ under~$\vct A$.
\end{lemma}

\begin{proof}
For the first claim, note that $\|\Proj_D(\vct{z})\|=\max_{\vct{y}\in D\cap B^{m}}\ip{\vct{z}}{\vct{y}}$, and $\max_{\vct{y}\in D\cap B^{m}}\ip{\vct{z}}{\vct{y}}=0$ is equivalent to $\ip{\vct{z}}{\vct{y}}\leq 0$ for all $\vct y\in D$, i.e.,~$\vct{z}\in D^{\polar}$.

For~\eqref{eq:A^(-1)(D^*)}, let $\vct x\in\vct A^{-1}(D^\polar)$ and $\vct y\in D$. 
Then $\langle \vct x,\vct A^T\vct y\rangle = \langle \vct{Ax},\vct y\rangle\leq 0$, as $\vct{Ax}\in D^\polar$. Therefore, $\vct A^{-1}(D^\polar)\subseteq (\vct A^TD)^\polar$. On the other hand, if $\vct v\in (\vct A^TD)^\polar$ and $\vct{y}\in D$, then $\langle\vct{Av},\vct y\rangle=\langle\vct v,\vct A^T\vct y\rangle\leq 0$, so that $\vct{Av}\in D^\polar$ and hence, $(\vct A^TD)^\polar\subseteq\vct A^{-1}(D^\polar)$.
\end{proof}

Recall from~\eqref{eq:def-nres,sres} that for $\vct A\in\IR^{m\times n}$, $C\subseteq\IR^n$ and $D\subseteq\IR^m$ closed convex cones, the restricted norm and singular value of~$\vct A$ are defined by
$\nres{A}{C}{D} := \max\{\|\Proj_D(\mtx{A}\vct{x})\|\mid \vct x\in C\cap S^{n-1}\}$ and $\sres{A}{C}{D} := \min\{\|\Proj_D(\mtx{A}\vct{x})\|\mid \vct x\in C\cap S^{n-1}\}$, respectively. The following proposition provides geometric conditions for the vanishing of the restricted norm or singular value.

\begin{proposition}\label{prop:nres=0,sres=0}
Let $\vct A\in\IR^{m\times n}$, $C\subseteq\IR^n$ and $D\subseteq\IR^m$ be closed convex cones. Then the restricted norm vanishes, $\nres{A}{C}{D} = 0$, if and only if $C\subseteq (\vct A^TD)^\polar$. Furthermore, the restricted singular value vanishes, $\sres{A}{C}{D} = 0$, if and only if $C\cap (\vct A^TD)^\polar\neq\{\vct0\}$, which is equivalent to $\vct AC\cap D^\polar\neq\{\vct0\}$ or $\ker\vct A\cap C\neq\{\vct0\}$.
\end{proposition}

\begin{proof}
Using Lemma~\ref{lem:A^(-1)(D^*)} we have $\Proj_D(\vct{Ax})=\vct0$ if and only if $\vct{Ax}\in D^\polar$. This shows $\nres{A}{C}{D} = 0$ if and only if $\vct{Ax}\in D^\polar$ for all $\vct x\in C\cap S^{n-1}$, or equivalently, $C\subseteq \vct A^{-1}(D^\polar)=(\vct A^TD)^\polar$ by~\eqref{eq:A^(-1)(D^*)}. The claim about the restricted singular value follows similarly: $\sres{A}{C}{D} = 0$ if and only if $\vct{Ax}\in D^\polar$ for some $\vct x\in C\cap S^{n-1}$, or equivalently, $C\cap \vct A^{-1}(D^\polar)\neq\{\vct0\}$. If $\vct x\in C\cap \vct A^{-1}(D^\polar)\setminus\{\vct0\}$, then either $\vct{Ax}$ is nonzero or $\vct x$ lies in the kernel of~$\vct A$, which shows the second characterization.
\end{proof}

It is easily seen that the restricted norm is symmetric $\nres{A}{C}{D} = \nrest{A}{D}{C}$,
\begin{align}\label{eq:symm-nres}
\begin{split}
  \nres{A}{C}{D} &= \max_{\vct x\in C\cap B^m}\max_{\vct y\in D\cap B^n} \langle \vct{Ax},
  \vct y\rangle\\
  &= \max_{\vct y\in D\cap B^n}\max_{\vct x\in C\cap B^m} \langle \vct A^T\vct y,\vct x\rangle \\
  &= \nrest{A}{D}{C} .
  \end{split}
\end{align}
Such a relation does not hold in general for the restricted singular value. 
In fact, in Section~\ref{sec:gen-feas-prob} we will see that, unless $C=D=\IR^n$, the minimum of~$\sres{A}{C}{D}$ and~$\srestm{A}{D}{C}$ is always zero, if~$C$ and~$D$ have nonempty interior, cf.~\eqref{eq:min(sres(A,C,D),srestm(A,D,C))=0}. And if~$C$ or~$D$ is a linear subspace then $\srestm{A}{D}{C}=\srest{A}{D}{C}$.

\begin{remark}
In the case $C=\IR^n$, $D=\IR^m$, with $m\geq n$, one can characterize the smallest singular value of~$\vct A$ as the inverse of the norm of the (Moore-Penrose) 
\emph{pseudoinverse} of~$\vct A$:
  \[ \sigma(\vct A) = \|\vct A^\dagger\|^{-1} . \]
Such a characterization does \emph{not} hold in general for the restricted singular value, i.e., in general one cannot write $\sres{A}{C}{D}$ as $\nresdag{A}{D}{C}^{-1}$. 
Consider for example the case $D=\IR^m$ and $C$ a circular cone of angle $\alpha$ around some center $\vct p\in S^{n-1}$. Both cones have nonempty interior, but 
letting $\alpha$ go to zero, it is readily seen that $\sres{A}{C}{D}$ tends to $\|\vct{Ap}\|$, while $\nresdag{A}{D}{C}$ tends to $\|\vct p^T\mtx{A}^\dagger\|$, which is in general 
not equal to~$\|\vct{Ap}\|^{-1}$, unless $\transp{\mtx{A}}\mtx{A}=\Id_n$.
\end{remark}

\subsection{The biconic feasibility problem}\label{sec:gen-feas-prob}
The convex feasibility problem in the setting with two nonzero closed convex cones $C\subseteq\IR^n$, $D\subseteq\IR^m$ is given as:
\\\def\tmpX{3mm}
\begin{align}
   \exists \vct x & \in C\setminus\{\vct0\} \quad\text{s.t.} \hspace{\tmpX} \vct{Ax} \in D^\polar ,
\tag{P}
\label{eq:(P)}\\
   \exists \vct y & \in D\setminus\{\vct0\} \quad\text{s.t.} \hspace{\tmpX} -\vct A^T\vct y\in C^\polar .
\tag{D}
\label{eq:(D)}
\end{align}
Using Lemma~\ref{lem:A^(-1)(D^*)} and Proposition~\ref{prop:nres=0,sres=0} we obtain the following characterizations of the primal feasible matrices $\P(C,D) := \{\vct A\in\IR^{m\times n}\mid \text{(P) is feasible}\}$,
\begin{align}\label{eq:P(C,D)}
\begin{split}
   \P(C,D) & \stackrel{\eqref{eq:A^(-1)(D^*)}}{=} \big\{\vct A\in\IR^{m\times n}\mid C\cap \big(\vct A^T D\big)^\polar\neq\{\vct0\}\big\}\\
   & \stackrel{\text{[Prop.~\ref{prop:nres=0,sres=0}]}}{=} \{\vct A\in\IR^{m\times n}\mid \sres{A}{C}{D}=0 \} .
   \end{split}
\end{align}
By symmetry, we obtain for the dual feasible matrices $\D(C,D) := \{\vct A\in\IR^{m\times n}\mid \text{(D) is feasible}\}$,
\begin{align}\label{eq:D(C,D)}
\begin{split}
   \D(C,D) & = \{\vct A\in\IR^{m\times n}\mid D\cap (-\vct A C)^\polar\neq\{\vct0\}\} \\
   &=  \{\vct A\in\IR^{m\times n}\mid \srestm{A}{D}{C}=0 \} .
   \end{split}
\end{align}
In fact, we will see that $\sres{A}{C}{D}$ and $\srestm{A}{D}{C}$ can be characterized as the distances to $\P(C,D)$ and $\D(C,D)$, respectively. 
We defer the proofs for this section to Appendix~\ref{sec:appendix-convex}.

In the following proposition we collect some general properties of $\P(C,D)$ and $\D(C,D)$.

\begin{proposition}\label{prop:primaldual}
Let $C\subseteq\IR^n$, $D\subseteq\IR^m$ be closed convex cones with nonempty interior. Then
\begin{enumerate}
  \item $\P(C,D)$ and $\D(C,D)$ are closed;
  \item the union of these sets is given by
     \[ \P(C,D)\cup\D(C,D) = \begin{cases} \{\vct A\in\IR^{m\times n}\mid \det\vct A=0\} & C=D=\IR^n \\ \IR^{m\times n} & \text{else} ; \end{cases} \]
\item the intersection $\P(C,D)\cap\D(C,D)$ is nonempty but has Lebesgue measure zero.
\end{enumerate}
\end{proposition}

Note that from~(2) and the characterizations~\eqref{eq:P(C,D)} and~\eqref{eq:D(C,D)} of $\P(C,D)$ and $\D(C,D)$, respectively, we obtain for every $\vct A\in\IR^{m\times n}$: $\min\{ \sres{A}{C}{D}, \srestm{A}{D}{C}\} = 0$ or, equivalently,
\begin{equation}\label{eq:min(sres(A,C,D),srestm(A,D,C))=0}
  \max\big\{ \sres{A}{C}{D}, \srestm{A}{D}{C}\big\} = \sres{A}{C}{D} + \srestm{A}{D}{C} ,
\end{equation}
unless $C=D=\IR^n$.

In the following we simplify the notation by writing $\P,\D$ instead of $\P(C,D),\D(C,D)$. For the announced interpretation of the restricted singular value as distance to $\P,\D$ we introduce the following notation: for $\vct A\in\IR^{m\times n}$ define
\begin{align*}
   \dist(\vct A,\P) & := \min\{\|\vct\Delta\|\mid \vct A+\vct\Delta\in\P\} , & \dist(\vct A,\D) & := \min\{\|\vct\Delta\|\mid \vct A+\vct\Delta\in\D\},
\end{align*}
where as usual, the norm considered is the operator norm. The proof of the following proposition, given in Appendix~\ref{sec:appendix-convex}, follows along the lines
of similar derivations in the case with a cone and a linear subspace~\cite{BF:09}.

\begin{proposition}\label{prop:sing-dist}
Let $C\subseteq\IR^n$, $D\subseteq\IR^m$ nonzero closed convex cones with nonempty interior. 
Then
\begin{align*}
  \dist(\vct A,\P) & =\sres{A}{C}{D} ,  \\
  \dist(\vct A,\D) & =\srestm{A}{D}{C} .
\end{align*}
\end{proposition}

We finish this section by considering the intersection of~$\P$ and~$\D$, which we denote by
  \[ \Sigma(C,D) := \P(C,D)\cap \D(C,D) , \]
or simply $\Sigma$ when the cones are clear from context.
This set is usually referred to as the set of \emph{ill-posed inputs}.
As shown in Proposition~\ref{prop:primaldual}, the set of ill-posed inputs, assuming~$C\subseteq\IR^n$ and~$D\subseteq\IR^m$ each have nonempty interior, is a nonempty zero volume set. In the special case $C=\IR^n$, $D=\IR^m$,
  \[ \Sigma(\IR^n,\IR^m) = \{\text{rank deficient matrices in } \IR^{m\times n}\} . \]
From~\eqref{eq:min(sres(A,C,D),srestm(A,D,C))=0} and Proposition~\ref{prop:sing-dist} we obtain, if $(C,D)\neq(\IR^n,\IR^n)$,
\begin{align*} 
 \dist(\vct A,\Sigma) & = \max\big\{\dist(\vct A,\P) , \dist(\vct A,\D)\big\} \\
 &= \dist(\vct A,\P) + \dist(\vct A,\D) . 
 \end{align*}

The inverse distance to ill-posedness forms the heart of Renegar's condition number~\cite{rene:94,rene:95b}. We denote
\begin{align}\label{eq:renegars}
\begin{split}
  \RCD{A}{C}{D} &:= \frac{\|\vct A\|}{\dist(\vct A,\Sigma(C,D))} \\
  &= \min\Bigg\{ \frac{\|\vct A\|}{\sres{A}{C}{D}} , \frac{\|\vct A\|}{\srestm{A}{D}{C}} \Bigg\} . 
  \end{split}
\end{align}
Furthermore, we abbreviate the special case $D=\IR^m$, which corresponds to the classical feasibility problem, by the notation
\begin{equation}\label{eq:renegars-single_cone}
  \Ren_C(\vct A) := \RCD{A}{C}{\IR^m} .
\end{equation}
Note that the usual matrix condition number is recovered in the case $C=\IR^n$, $D=\IR^m$,
  \[ \Ren_{\IR^n}(\vct A) = \Ren_{\IR^n,\IR^m}(\vct A) = \kappa(\vct A) . \]
Another simple but useful property is the symmetry $\Ren_{C,D}(\vct A) = \Ren_{D,C}(-\vct A^T)$. Finally, note that the restricted singular value has the following monotonicity properties
\begin{align*}
   C\subseteq C' & \Rightarrow \sres{A}{C}{D}\geq \sres{A}{C'}{D} ,\\
    D\subseteq D' & \Rightarrow \sres{A}{C}{D}\leq \sres{A}{C}{D'} .
\end{align*}
This indicates that not necessarily $\Ren_C(\vct A)\leq\Ren_{C'}(\vct A)$ if $C\subseteq C'$. But in the case $C'=\IR^n$ and $m\geq n$ this inequality does hold, which we formulate in the following lemma.

\begin{lemma}
Let $C\subseteq\IR^n$ closed convex cone with nonempty interior and $\vct A\in\IR^{m\times n}$ with $m\geq n$. Then
\begin{equation}\label{eq:R_C(A)<=kappa(A)}
  \Ren_C(\vct A) \leq \kappa(\vct A) .
\end{equation}
\end{lemma}

\begin{proof}
In the case $C=\IR^n$ we have $\Ren_{\IR^n}(\vct A) = \kappa(\vct A)$. If $C\neq\IR^n$ then $\vct AC\neq\IR^m$, as $m\geq n$. It follows that $\IR^m\cap (-\vct AC)^\polar\neq\{\vct0\}$, and thus $\srestm{A}{\IR^m}{C}=0$, cf.~\eqref{eq:D(C,D)}. Hence,
  \[ \Ren_C(\vct A)=\frac{\|\vct A\|}{\sres{A}{C}{\IR^m}}\leq \frac{\|\vct A\|}{\sres{A}{\IR^n}{\IR^m}} = \kappa(\vct A) . \qedhere \]
\end{proof}

To conclude this section, we state a useful bound on the condition number of a product of matrices. 

\begin{lemma}\label{le:product_bound}
Let $\mtx{A}\in \R^{m\times n}$ with $m\leq n$ and let $\mtx{B}\in \R^{m\times m}$ be nonsingular. Then
\begin{equation*}
  \mathcal{R}_{C}(\mtx{B}\mtx{A}) \leq \kappa(\mtx{B})\cdot \mathcal{R}_C(\mtx{A}).
\end{equation*}
\end{lemma}

\begin{proof}
We need to bound the numerator from above and the denominator from below in the definition of Regenar's condition number~\eqref{eq:renegars}. 
For the norms we have $\|\mtx{BA}\|\leq \|\mtx{B}\|\cdot \|\mtx{A}\|$.
If $\sigma_C(\mtx{BA})=\sigma_{\R^m\to C}(-\mtx{A}^T\mtx{B}^T)=0$, then clearly also $\sigma_C(\mtx{A})=\sigma_{\R^m\to C}(-\mtx{A}^T)=0$. 
Assume that $\sigma_C(\mtx{BA})\neq 0$, and let $\vct{x}\in C\cap S^{n-1}$. Since $\mtx{B}$ is non-singular, $\mtx{A}\vct{x}\neq \zerovct$ and set $\vct{z}=\mtx{Ax}/\|\mtx{Ax}\|$. Then
\begin{equation*}
  \|\mtx{BA}\vct{x}\| = \|\mtx{B}\vct{z}\|\cdot \|\mtx{A}\vct{x}\| \geq \sigma(\mtx{B})\cdot \sigma_C(\mtx{A}\vct{x})\neq 0.
\end{equation*}
If $\sigma_{\R^m\to C}(-\mtx{A}^T\mtx{B}^T)\neq 0$, then if $\vct{x}\in S^{m-1}$ and $\vct{z}=\mtx{B}^T\vct{x}/\|\mtx{B}^T\vct{x}\|$, then
\begin{align*}
\|\Proj_C(\mtx{A}^T\mtx{B}^T\vct{x})\| &=\|\Proj_C(\mtx{A}^T\vct{z})\|\cdot \|\mtx{B}^T\vct{x}\| \\
&\geq \sigma(\mtx{B}) \cdot \sigma_{\R^m\to C}(-\mtx{A}^T)\neq 0.
\end{align*}
The condition bound follows.
\end{proof}

\section{Linear images of cones}\label{sec:lin_imag}
The norm of the projection is a special case of a cone-restricted norm:
\begin{equation}\label{eq:normproj-resnorm}
 \|\Proj_C(\vct g)\| = \norm{\mtx{g}}_{\R_+\to C},
\end{equation}
where on the right-hand side we interpret $\mtx{g}\in \R^{n\times 1}$ as linear map from $\R$ to $\R^n$. In this section we relate these norms for linear images of convex cones. The upper bound in Theorem~\ref{thm:moments-res.norm} is a special case of
a more general bound for moment functionals~\cite[Proposition 3.9]{amelunxen2014gordon}.

\begin{theorem}\label{thm:moments-res.norm}
Let $C\subseteq\IR^n$ be a closed convex cone, and
$\nu_r(C) := \Expect[\norm{\Proj_C(\vct{g})}^r]$,
with $\vct{g}\in\IR^n$ Gaussian. Then for $\vct A\in\IR^{p\times n}$, and $r\geq1$,
\begin{equation}\label{eq:mu_r(TC)<=...}
 \nu_r(\mtx{A}C) \leq \Ren_C(\mtx{A})^r\nu_r(C).
\end{equation}
In particular, if $p\geq n$ and $\mtx{A}$ has full rank, then
\begin{equation}\label{eq:mu_r(TC)<=...-kappa}
  \frac{\sdim(C)}{\kappa(\vct A)^2} \leq \sdim(\vct AC) \leq \kappa(\vct A)^2\, \sdim(C).
\end{equation}
\end{theorem}

The proof of Theorem~\ref{thm:moments-res.norm} relies on the following auxiliary result, Lemma~\ref{lem:TC-1}, and on a generalized form of Slepian's inequality, Theorem~\ref{lem:Slepian}.

\begin{lemma}\label{lem:TC-1}
Let $C\subseteq \IR^n$ be a closed convex cone and $\mtx{A}\in \R^{p\times n}$. Then
\begin{equation}
  \frac{1}{\norm{\mtx{A}}} \mtx{A}(C\cap B^{n}) \subseteq \vct AC\cap B^p \,\subseteq\, \tfrac{1}{\lambda}\, \vct A(C\cap B^n) ,
\end{equation}
with $\lambda := \max\big\{\sres{A}{C}{\IR^p},\srestm{A}{\IR^p}{C}\big\}$.
\end{lemma}

\begin{proof}
For the lower inclusion, note that any $\vct{y}\in \frac{\mtx{A}(C\cap B^n)}{\norm{\mtx{A}}}$ can be written as $\vct{y}=\frac{\mtx{A}\vct{x}}{\norm{\mtx{A}}}$, with $\vct{x}\in C\cap B^n$. Since $\norm{\mtx{A}\vct{x}}\leq \norm{\mtx{A}}$, we have
$\vct{y} \in \conv\left\{\zerovct,\frac{\mtx{A}\vct{x}}{\norm{\mtx{A}\vct{x}}}\right\} \subset (\mtx{A}C)\cap B^{p}$.
which was to be shown.

For the upper inclusion, let $\lambda_1:=\sres{A}{C}{\IR^p}$, $\lambda_2:=\srestm{A}{\IR^p}{C}$. We show in two steps that $\vct AC\cap B^p \,\subseteq\, \tfrac{1}{\lambda_1}\, \vct A(C\cap B^n)$ if $\lambda_1>0$ and $\vct AC\cap B^p \,\subseteq\, \tfrac{1}{\lambda_2}\, \vct A(C\cap B^n)$ if $\lambda_2>0$.

(1) Let $\lambda_1>0$. Since $\mtx{A}\vct{C} \cap B^p$ as well as $\vct A(C\cap B^n)$ contain the origin, it suffices to show that $\vct AC\cap S^{p-1}\subseteq \tfrac{1}{\lambda_1}\, \vct A(C\cap B^n)$. Every element in $\vct AC\cap S^{p-1}$ can be written as $\frac{\vct{Ay}_0}{\|\vct{Ay}_0\|}$ for some $\vct y_0\in C\cap S^{n-1}$, and since $\sres{A}{C}{\IR^p}=\min_{\vct y\in C\cap S^{n-1}}\|\vct{Ay}\|\leq \|\vct{Ay}_0\|$, we obtain $\sres{A}{C}{\IR^p}\frac{\vct{Ay}_0}{\|\vct{Ay}_0\|}\in \conv\{\vct0,\vct{Ay}_0\}\subseteq\vct A(C\cap B^n)$. This shows $\vct AC\cap S^{p-1}\subseteq \tfrac{1}{\lambda_1}\, \vct A(C\cap B^n)$.

(2) Let $\lambda_2>0$. Recall from~\eqref{eq:D(C,D)} that $\lambda_2=\srestm{A}{\IR^p}{C}>0$ only if $(\vct AC)^\polar=\{\vct0\}$, i.e., $\vct AC=\IR^p$. Observe that
\begin{align*}
   \srestm{A}{\IR^p}{C} & = \min_{\vct z\in S^{p-1}} \max_{\vct y\in C\cap B^n} \langle \vct{Ay},\vct z\rangle \\
   & = \max\big\{r\geq0\mid r B^p\subseteq \vct A(C\cap B^n)\big\} .
\end{align*}
This shows $B^p \subseteq \tfrac{1}{\lambda_2}\, \vct A(C\cap B^n)$ and thus finishes the proof.
\end{proof}

The following generalization of Slepian's inequality is the special case of a generalized version of Gordon's inequality for Gaussian processes,~\cite[Theorem B.2]{amelunxen2014gordon}, when setting $m=1$ in that theorem.

\begin{theorem}\label{lem:Slepian}
Let $X_{j},Y_{j}$, $j\in \{0,\dots,n\}$, be centered Gaussian random variables, and assume that for all $j,k\geq 0$ we have
\begin{equation*}
  \Expect |X_{j}-X_{k}|^2 \leq \Expect |Y_{j}-Y_{k}|^2.
\end{equation*}
Then for any monotonically increasing convex function $f\colon\IR_+\to\IR$,
\begin{equation}\label{eq:E[f(minmax(0,X_(ij)-X0))]>=...}
  \Expect \max_j f_+(X_{j}-X_0) \leq \Expect \max_j f_+(Y_{j}-Y_0) ,
\end{equation}
where $f_+(x):=f(x)$, if $x\geq0$, and $f_+(x):=f(0)$, if $x\leq0$.
\end{theorem}

\begin{proof}[Proof of Theorem~\ref{thm:moments-res.norm}]
Set $\lambda := \max\big\{\sres{A}{C}{\IR^p},\srestm{A}{\IR^p}{C}\big\}$.
For the upper bound, note that by Lemma~\ref{lem:TC-1} we have
\begin{equation*}
  \Expect[\norm{\Proj_{\mtx{A}C}(\vct{g})}^r] = \Expect\left[\max_{\vct{x}\in \mtx{A}C\cap B^p}\ip{\vct{g}}{\vct{x}}^r\right]\leq \frac{1}{\lambda^r} \Expect\left[\max_{\vct{x}\in C\cap B^n}\ip{\vct{g}}{\mtx{A}\vct{x}}^r\right].
\end{equation*}
Let $\vct{g}$ be a standard Gaussian vector and consider the Gaussian processes $X_{\vct{x}} = \ip{\vct{g}}{\mtx{A}\vct{x}}$ and $Y_{\vct{x}}=\ip{\vct{g}}{\norm{\mtx{A}}\vct{x}}$, indexed by $\vct{x}\in C\cap B^n$. For any $\vct{x},\vct{y}\in C\cap B^n$ we have
\begin{equation*}
\Expect(X_{\vct{x}}-X_{\vct{y}})^2 = \|\mtx{A}\vct{x}-\mtx{A}\vct{y}\|^2 \leq  \norm{\norm{\mtx{A}}\vct{x}-\norm{\mtx{A}}\vct{y}}^2 = \Expect(Y_{\vct{x}}-Y_{\vct{y}})^2,
\end{equation*}
we get $\Expect(X_{\vct{x}}-X_{\vct{y}})^2 \leq \Expect(Y_{\vct{x}}-Y_{\vct{y}})^2$. 
From Theorem~\ref{lem:Slepian} we conclude that for any finite subset $S\subset C\cap B^n$ containing the origin,
\begin{equation*}
  \Expect[\max_{\vct{x}\in S} X_{\vct{x}}^r] \leq \Expect[\max_{\vct{x}\in S} Y_{\vct{x}}^r].
\end{equation*}
By a standard compactness argument (see, e.g.,~\cite[8.6]{FR:13}), this extends to the whole index set $C\cap B^n$, which yields the inequalities
\begin{align*}
 \nu_r(\mtx{A}C) &=  \Expect[\norm{\Proj_{\mtx{A}C}(\vct{g})}^r] \\
 &\leq \frac{1}{\lambda^r} \Expect\left[\max_{\vct{x}\in C\cap B^n}\ip{\vct{g}}{\mtx{A}\vct{x}}^r\right] \\
 & \leq \frac{\norm{\mtx{A}}^r}{\lambda^r} \Expect\left[\max_{\vct{x}\in C\cap B^n}\ip{\vct{g}}{\vct{x}}^r\right] \\
 &= \Ren_C(\mtx{A})^r \nu_r(C).
\end{align*}
The upper bound in terms of the usual matrix condition number follows courtesy of~\eqref{eq:R_C(A)<=kappa(A)}.
The lower bound proceeds along the lines, with the roles of $\norm{\mtx{A}}$ and $\lambda$ reversed. More specifically, from Lemma~\ref{lem:TC-1} we get the inequality
\begin{equation*}
  \Expect[\norm{\Proj_{\mtx{A}C}(\vct{g})}^r] \geq \frac{1}{\norm{\mtx{A}}^r} \Expect\left[\max_{\vct{x}\in C\cap B^{n}}\ip{\vct{g}}{\mtx{A}\vct{x}}^r\right].
\end{equation*}
Define 
\begin{equation*}
  \sigma_{C-C}(\mtx{A}) = \min_{\vct{z}\in S(C-C)} \|\mtx{A}\vct{z}\|,
\end{equation*}
where $S(C-C):=\{(\vct{x}-\vct{y})/\|\vct{x}-\vct{y}\|\mid \vct{x}\in C\cap B^n, \vct{y}\in C\cap B^n, \vct{x}\neq \vct{y}\}$. 
Consider the processes $Y_{\vct{x}} = \ip{\vct{g}}{\mtx{A}\vct{x}}$ and $X_{\vct{x}}=\ip{\vct{g}}{\sigma_{C-C}(\mtx{A})\vct{x}}$ indexed by $\vct{x}\in C\cap B^{n}$. Then for distinct $\vct{x},\vct{y}\in C\cap B^n$,
\begin{align*}
  \Expect(X_{\vct{x}}-X_{\vct{y}})^2 &= \|\sigma_{C-C}(\mtx{A})\vct{x}-\sigma_{C-C}(\mtx{A})\vct{y}\|^2 \\
  &\leq  \norm{\mtx{A}\vct{x}-\mtx{A}\vct{y}}^2 = \Expect(Y_{\vct{x}}-Y_{\vct{y}})^2.
\end{align*}
We can now apply Slepian's inequality as we did for the upper bound, and conclude that
\begin{equation*}
\Expect[\norm{\Proj_{\mtx{A}C}(\vct{g})}^r] \geq \frac{\sigma_{C-C}(\mtx{A})^r}{\norm{\mtx{A}}^r} \nu_r(C).
\end{equation*}
To finish the argument, note that we have $\sigma_{C-C}(\mtx{A})\geq \sigma(\mtx{A})$. 
\end{proof}

\section{Conic integral geometry}\label{sec:conic-integral-geometry}
In this section we use integral geometry to develop the tools needed for deriving a preconditioned bound in Theorem~\ref{thm:B}. A comprehensive treatment of integral geometry can be found in~\cite{SW:08}, while a self-contained treatment in the setting of polyhedral cones, which uses our language, is given in~\cite{amelunxen2017intrinsic}. 

\subsection{Intrinsic volumes}\label{sec:intrinsic-volumes}
The theory of conic integral geometry is based on the \emph{intrinsic volumes} $v_0(C),\ldots,v_n(C)$ of a closed convex cone $C\subseteq\IR^n$. The intrinsic volumes form a discrete probability distribution on $\{0,\ldots,n\}$ that capture statistical properties of the cone~$C$. For a polyhedral cone $C$ and $0\leq k\leq n$, the intrinsic volumes can be defined as
\begin{equation*}
v_k(C) = \Prob\{\Proj_C(\vct{g}) \in \relint(F), \quad \dim F=k\},
\end{equation*}
where $F$ is a face of $C$ and $\relint$ denotes the relative interior.

\begin{example}
 Let $C=L\subseteq \R^n$ be a linear subspace of dimension $i$. Then
 \begin{equation*}
  v_k(C) = \begin{cases}
            1 & \text{ if } k=i,\\
            0 & \text{ if } k\neq i.
           \end{cases}
 \end{equation*}
\end{example}

\begin{example}
 Let $C=\R^n_{\geq 0}$ be the non-negative orthant, i.e., the cone consisting of points with non-negative coordinates. A vector $\vct{x}$ projects orthogonally to a $k$-dimensional face of $C$ if and only if exactly $k$ coordinates are non-positive. By symmetry considerations and the invariance of the Gaussian distribution under permutations of the coordinates, it follows that
 \begin{equation*}
  v_k(\R^n_{\geq 0}) = \binom{n}{k}2^{-n}.
 \end{equation*}
\end{example}

For non-polyhedral closed convex cones, the intrinsic volumes can be defined by polyhedral approximation. To avoid having to explicitly take care of upper summation bounds in many formulas, we use the convention that $v_k(C)=0$ if $C\subseteq \R^n$ and $k>n$ (that this is not just a convention follows from the fact that intrinsic volumes are ``intrinsic'', i.e., not dependent on the dimension of the space in which $C$ lives).

The following important properties of the intrinsic volumes, which are easily verified in the setting of polyhedral cones, will be used frequently:
\begin{itemize}
 \item[(a)] {\bf Orthogonal invariance.} For an orthogonal transformation $\vct Q\in O(n)$,
 \begin{equation*}
  v_k(\vct QC) = v_k(C);
 \end{equation*}
 \item[(b)] {\bf Polarity.} 
 \begin{equation*}
  v_k(C) = v_{n-k}(C^{\polar});
 \end{equation*}
 \item[(c)] {\bf Product rule.} 
 \begin{equation}\label{eq:productrule}
  v_k(C\times D) = \sum_{i+j=k} v_i(C)v_j(D).
 \end{equation}
 In particular, if $D=L$ is a linear subspace of dimension $j$, then $v_{k+j}(C\times L) = v_{k}(C)$.
 \item[(d)] {\bf Gauss-Bonnet.}
 \begin{equation}\label{eq:gaussbonnet}
 \sum_{k=0}^n (-1)^k v_k(C) = \begin{cases}
 0 & \text{ if $C$ is not a linear subspace,}\\
 1 & \text{ else.}
 \end{cases}
 \end{equation}
\end{itemize}

\begin{figure}[h!]
\centering
\includegraphics[width=0.4\textwidth]{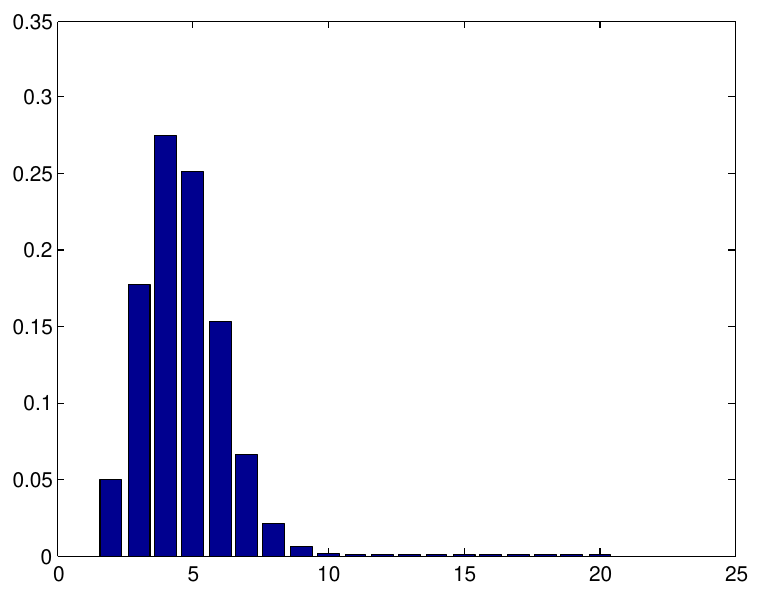}
\caption{Intrinsic volumes of the cone $C=\{\vct{x}\mid x_1\leq \cdots \leq x_n\}$.} 
\end{figure}

\subsection{The statistical dimension}\label{sec:stat-dim}
In what follows it will be convenient to work with reparametrizations of the intrinsic volumes, namely the tail and half-tail functionals
\begin{equation*}
  t_k(C) = \sum_{i\geq 0} v_{k+i}(C), \quad \quad \quad h_k(C) = 2\sum_{i\geq 0 \text{ even}} v_{k+i}(C),
\end{equation*}
which are defined for $0\leq k\leq n$. Adding (or subtracting) the Gauss-Bonnet relation~\eqref{eq:gaussbonnet} to the identity $\sum_{i\geq 0}v_i(C)=1$, we see that $h_0(C)=h_1(C)=1$ if $C$ is not a linear subspace, so that the sequences $2v_0(C),2v_2(C),\dots$ and $2v_1(C),2v_3(C),\dots$ are probability distributions in their own right. Moreover, we have the interleaving property
\begin{equation*}
  t_{i+1}(C) \leq h_i(C) \leq t_{i}(C).
\end{equation*}
The intrinsic volumes can be recovered from the half-tail functionals as
\begin{equation}\label{eq:v-intermsof-h}
 v_i(C) = \begin{cases} \frac{1}{2}(h_i(C)-h_{i+2}(C)) & \text{ for } 0\leq i\leq n-2,\\
  \frac{1}{2}h_{i}(C) &  \text{ else.}
  \end{cases}
\end{equation}
An important summary parameter is the {\em statistical dimension} of a cone $C$, defined as the expected value of the intrinsic volumes considered as probability distribution:
\begin{equation*}
  \delta(C) = \sum_{k=0}^n k v_k(C) =  \frac{1}{2}h_1(C)+\sum_{i\geq 2} h_i(C).
\end{equation*}

The statistical dimension coincides with the expected squared norm of the projection of a Gaussian vector on the cone, $\delta(C) = \Expect\big[\norm{\Proj_C(\vct{g})}^2\big]$. Moreover, it differs from the squared Gaussian width by at most $1$,
\begin{equation*}
  w^2(C) \leq \delta(C) \leq w^2(C)+1,
\end{equation*}
see~\cite[Proposition 10.2]{edge}.

The statistical dimension reduces to the usual dimension for linear subspaces, and also extends various properties of the dimension to closed convex cones $C\subseteq \R^n$:

\begin{itemize}
 \item[(a)] {\bf Orthogonal invariance.} For an orthogonal transformation $\vct Q\in O(n)$,
 \begin{equation*}
  \delta(\vct QC) = \delta(C);
 \end{equation*}
 \item[(b)] {\bf Complementarity.} 
 \begin{equation}\label{eq:complement}
  \delta(C) + \delta(C^{\polar}) = n;
 \end{equation}
 This generalizes the relation $\dim L+\dim L^{\perp} = n$ for a linear subspace $L\subseteq \R^n$.
 \item[(c)] {\bf Additivity.} 
 \begin{equation*}
  \delta(C\times D) = \delta(C)+\delta(D).
 \end{equation*}
 \item[(d)] {\bf Monotonicity.} 
 \begin{equation*}
  \delta(C)\leq \delta(D) \text{ if } C\subseteq D.
 \end{equation*}
\end{itemize}

The analogy with linear subspaces will be taken further when discussing concentration of intrinsic volumes, see Section~\ref{sec:concentration}.

\subsection{The kinematic formulas}\label{sec:kinematic}
The intrinsic volumes allow to study the properties of random intersections of cones via the {\em kinematic formulas}. A self-contained proof of these formulas for polyhedral cones is given in~\cite[Section 5]{amelunxen2017intrinsic}. In what follows, when we say that $\mtx{Q}$ is drawn uniformly at random from the orthogonal group $O(d)$, we mean that it is drawn from the Haar probability measure $\nu$ on $O(n)$. This is the unique regular Borel measure on $O(n)$ that is left and right invariant ($\nu(\vct{Q}A)=\nu(A\vct{Q})=\nu(A)$ for $\vct{Q}\in O(n)$ and a Borel measurable $A\subseteq O(n)$) and satisfies $\nu(O(n))=1$. Moreover, for measurable $f\colon O(n)\to \R_+$, we write
\begin{equation*}
 \Expect_{\vct{Q}\in O(n)}[f(\vct{Q})] := \int_{\vct{Q}\in O(n)} f(\vct{Q}) \ \nu(\diff{\mtx{Q}})
\end{equation*}
for the integral with respect to the Haar probability measure, and we will occasionally omit the subscript $\vct{Q}\in O(n)$, or just write $\mtx{Q}$ in the subscript, when there is no ambiguity. 

\begin{theorem}[Kinematic Formula]\label{thm:kinematic}
Let $C,D\subseteq \R^n$ be polyhedral cones.
Then, for $\vct Q\in O(n)$ uniformly at random, and $k>0$,
\begin{align}
\label{eq:kinematic}
  \Expect[v_k(C\cap \vct QD)] & = v_{k+n}(C\times D) , & \Expect[v_0(C\cap \vct QD)] & = t_0(C\times D) .
\intertext{If $D=L$ is a linear subspace of dimension $n-m$, then}
\label{eq:crofton}
  \Expect[v_k(C\cap \vct QL)] & = v_{k+m}(C) , & \Expect[v_0(C\cap \vct QL)] & =  \sum_{j=0}^m v_j(C) .
\end{align}
\end{theorem}
Combining Theorem~\ref{thm:kinematic} with the Gauss-Bonnet relation~\eqref{eq:gaussbonnet} yields the so-called {\em Crofton formulas}, which we formulate in the following corollary. The intersection probabilities are also know as Grassmann angles in the literature (see~\cite[2.33]{amelunxen2017intrinsic} for a discussion and references).

\begin{corollary}\label{cor:Croft}
Let $C,D\subseteq \R^{n}$ be polyhedral cones such that not both of $C$ and $D$ are linear subspaces, and let $L\subset \R^n$ be a linear subspace of dimension $n-m$.  Then, for $\vct Q\in O(n)$ uniformly at random,
\begin{align*}
  \Prob\{C\cap \vct QD \neq \zerovct\} &= h_{n+1}(C\times D), &
  \Prob\{C\cap \vct QL \neq \zerovct\} &= h_{m+1}(C).
\end{align*}
\end{corollary}

Applying the polarity relation $(C\cap D)^{\circ}=C^{\circ}+D^{\circ}$ (see~\cite[Proposition 2.5]{amelunxen2017intrinsic}) to the kinematic formulas, we obtain a polar version of the kinematic formula, for $k>0$,
\begin{align}\label{eq:kinematic-polar}
  \Expect[v_{n-k}(C+\vct QD)] & = v_{n-k}(C\times D) , & \Expect[v_n(C+\vct QD)] & = t_{n}(C\times D).
\end{align}

A convenient consequence of this polar form is a projection formula for intrinsic volumes, due to Glasauer~\cite{Gl}. Let $\mtx{Q}\in O(n)$ uniform at random and $\mtx{P}\in \R^{n\times n}$ a fixed orthogonal projection onto a linear subspace $L$ of dimension $m$. Then for $0<k\leq m$,
\begin{align}\label{eq:projection} 
\Expect[v_{m-k}(\mtx{PQ}C)] & = v_{m-k}(C),  & \Expect[v_m(\mtx{PQ}C)] & = t_m(C).
\end{align}

As we will see in Section~\ref{sec:tqc}, this results holds for {\em any} full rank $\mtx{T}\in \R^{m\times n}$, instead of just for projections $\mtx{P}$.

\begin{remark} The astute reader may notice that the projection $\mtx{PQ}C$ does not need to be a closed convex cone. For random $\mtx{Q}$, however, the probability of this happening can be shown to be zero.
\end{remark}

\subsection{Concentration of measure}\label{sec:concentration}
It was shown in~\cite{edge} (with a more streamlined and improved derivation in~\cite{mccoy2014steiner}), that the intrinsic volumes concentrate sharply around the statistical dimension.
For a closed convex cone $C$, 
let $X_C$ denote the discrete random variable satisfying 
\begin{equation*}
 \Prob\{X_C=k\} = v_k(C).
\end{equation*}
The following result is from~\cite{mccoy2014steiner}.

\begin{theorem}\label{thm:concentration}
Let $\lambda\geq 0$. Then
\begin{equation*}
  \Prob\{|X_C-\delta(C)|\geq \lambda\} \leq 2\exp\left(\frac{-\lambda^2/4}{\min\{\delta(C),\delta(C^{\circ})\}+\lambda/3}\right).
\end{equation*}
\end{theorem}

Roughly speaking, the intrinsic volumes of a convex cone in high dimensions approximate those of a linear subspace of dimension $\delta(C)$. The concentration result~\ref{thm:concentration}, used in conjunction with the kinematic formula, gives rise to an approximate kinematic formula, which in turn underlies the phase transition results from~\cite{edge}. We will only need the following direct consequence of Theorem~\ref{thm:concentration}.

\begin{corollary}\label{col:approx-kinematic}
Let $\eta\in (0,1)$, let $C$ be a closed convex cone, and let $0\leq m\leq n$. Then 
 \begin{align*}
 \delta(C) \leq m-a_{\eta}\sqrt{m} \ &\Longrightarrow \ t_{m} \leq \eta;\\
  \delta(C)  \geq m+a_{\eta}\sqrt{m} \ &\Longrightarrow \ t_{m} \geq 1- \eta,
  \end{align*}
  with $a_{\eta}:=2\sqrt{\log(2/\eta)}$. 
\end{corollary}

Applying the above to the statistical dimension, we get the following expression.

\begin{corollary}\label{cor:pqc}
Let $\eta\in (0,1)$ and assume that $m\geq \delta(C)+a_{\eta}\sqrt{m}$, with $a_{\eta}=2\sqrt{\log(2/\eta)}$. Then
\begin{equation*}
  \delta(C) - (n-m) \eta \leq \Expect_{\mtx{Q}}[\delta(\mtx{PQ}C)] \leq  \delta(C).
\end{equation*}
\end{corollary}

\begin{proof}
A direct application of the projection formulas~\eqref{eq:projection} and the definition of the statistical dimension shows that
\begin{equation*}
  \Expect_{\mtx{Q}}[\delta(\mtx{PQ}C)] = \delta(C) - \sum_{k=1}^{n-m} k v_{k+m}(C).
\end{equation*}
The bound then follows by bounding the right-hand side in a straight-forward way and applying Corollary~\ref{col:approx-kinematic}.
\end{proof}

\subsection{The TQC Lemma}\label{sec:tqc}
The following generalization of the projection formulas~\eqref{eq:projection}, first observed by Mike McCoy and Joel Tropp, may at first sight look surprising. While it can be deduced from general integral-geometric considerations (see, for example, ~\cite{amelunxen2014measures}), we include a proof because it is illustrative.

\begin{lemma}\label{le:tqc} Let $\mtx{T}\in \R^{m\times n}$ be of full rank. Then for $0\leq k<m$, 
\begin{align}\label{eq:tqc}
  \Expect[v_k(\mtx{T}\mtx{Q}C)] &= v_k(C) , & \Expect[v_m(\mtx{T}\mtx{Q}C)] &= t_m(C)
\end{align}
\end{lemma}

\begin{proof}
In view of~\eqref{eq:v-intermsof-h}, it suffices to show~\eqref{eq:tqc} for the half-tail functionals $h_j$ instead of the intrinsic volumes $v_j$. Let $L\subset\R^n$ be a linear subspace of dimension $\dim L=k\leq m$. 
From Proposition~\ref{prop:nres=0,sres=0} it follows that
\begin{equation*}
  \mtx{Q}C\cap \mtx{T}^{-1}L\neq \{\zerovct\} \Longleftrightarrow \mtx{TQ}C\cap L \neq \{\zerovct\} \text{ or } \ker \mtx{T}\cap \mtx{Q}C \neq \{\zerovct\},
\end{equation*} 
where in this case, as before, $\mtx{T}^{-1}L$ denotes the pre-image of $L$ under $\mtx{T}$. 
Denoting by $\mtx{P}$ the orthogonal projection onto the complement $(\ker\mtx{T})^{\perp}$, 
we thus get
\begin{equation*}
  \mtx{PQ}C\cap (\mtx{T}^{-1}L\cap (\ker\mtx{T})^{\perp}) \neq \{\zerovct\} \Longleftrightarrow \mtx{TQ}C\cap L \neq \{\zerovct\},
\end{equation*}
and taking probabilities,
\begin{equation}\label{eq:probabilities}
\Prob\left\{\mtx{PQ}C\cap (\mtx{T}^{-1}L\cap (\ker\mtx{T})^{\perp}) \neq \{\zerovct\}\right\} = \Prob\{\mtx{TQ}C\cap L \neq \{\zerovct\}\}.
\end{equation}
To compute the probability on the left, let $\mtx{Q}_0$ is a random orthogonal transformation of the space $(\ker \mtx{T})^{\perp}$. Restricting to $(\ker \mtx{T})^{\perp}$ as ambient space,
\begin{align*}
 \Prob_{\mtx{Q}}&\left\{\mtx{PQ}C\cap (\mtx{T}^{-1}L\cap (\ker\mtx{T})^{\perp}) \neq \{\zerovct\}\right\} \\
 &= \Prob_{\mtx{Q}}\left\{\mtx{PQ}C\cap \mtx{Q}_0(\mtx{T}^{-1}L\cap (\ker\mtx{T})^{\perp}) \neq \{\zerovct\}\right\}\\
 &= \Expect_{\mtx{Q}_0}\Prob_{\mtx{Q}}\left\{\mtx{PQ}C\cap \mtx{Q}_0(\mtx{T}^{-1}L\cap (\ker\mtx{T})^{\perp}) \neq \{\zerovct\}\right\}\\
 &\stackrel{(1)}{=} \Expect_{\mtx{Q}}\Prob_{\mtx{Q}_0}\left\{\mtx{PQ}C\cap \mtx{Q}_0(\mtx{T}^{-1}L\cap (\ker\mtx{T})^{\perp}) \neq \{\zerovct\}\right\}\\
 &\stackrel{(2)}{=} \Expect_{\mtx{Q}}[h_{m-k+1}(\mtx{PQ}C)]
\end{align*}
where for (1) we summoned Fubini on the representation of the probability as expectation of an indicator variable and for (2) the Crofton formula~\ref{cor:Croft} with $(\ker\mtx{T})^{\perp}$ as ambient space. 
A similar argument on the right-hand side of~\eqref{eq:probabilities} shows that
\begin{equation*}
\Prob_{\mtx{Q}}\{\mtx{TQ}C\cap L \neq \{\zerovct\}\} = \Expect_{\mtx{Q}}[h_{m-k+1}(\mtx{TQ}C)].
\end{equation*}
In summary, we have for shown that $\Expect_{\mtx{Q}}[h_{m-k+1}(\mtx{TQ}C)] = \Expect_{\mtx{Q}}[h_{m-k+1}(\mtx{PQ}C)]$ for $0\leq k\leq m$, and hence also $\Expect_{\mtx{Q}}[v_i(\mtx{TQ}C)] = \Expect_{\mtx{Q}}[v_i(\mtx{PQ}C)]$ for $0\leq i\leq m$. The claim now follows by applying the projection formula~\eqref{eq:projection}.
\end{proof}

As with the case where $\mtx{T}$ is a projection, applying the above to the statistical dimension, we get the following expression.

\begin{corollary}
Let $\eta\in (0,1)$ and assume that $m\geq \delta(C)+a_{\eta}\sqrt{m}$, with $a_{\eta}=2\sqrt{\log(2/\eta)}$. Then under the conditions of Lemma~\ref{le:tqc}, we have
\begin{equation*}
  \delta(C) - (n-m) \eta \leq \Expect_{\mtx{Q}}[\delta(\mtx{TQ}C)] \leq  \delta(C) - \eta.
\end{equation*}
\end{corollary}

It remains to be seen whether the fact that the main preconditionining results can be formulated with an arbitrary matrix $\mtx{T}$, rather than just a projection $\mtx{P}$, can be of use.

\section{Improved condition bounds}\label{sec:improved}
In this section we derive the improved condition number bounds on the statistical dimension. We first derive Corollary~\ref{prop:improved-cond}, restated here as a proposition, which is a simple consequence of the behaviour of the statistical dimension under polarity.

\begin{proposition}
Let $C\subseteq\IR^n$ be a closed convex cone, and
$\delta(C)$ the statistical dimension of $C$. Then for $\vct A\in\IR^{n\times n}$ of full rank,
\begin{equation*}
  \sdim(\vct AC) \leq \kappa(\vct A)^{-2} \cdot \sdim(C) + \left(1-\kappa(\vct A)^{-2}\right)\cdot n.
\end{equation*}
\end{proposition}

\begin{proof}
We have
\begin{align*}
  \delta(\mtx{A}C) &\stackrel{(1)}{=} n-\delta(\mtx{A}^{-T}C^{\circ})\\
  &\stackrel{(2)}{\leq} n-\kappa(\mtx{A})^{-2} \delta(C^{\circ}) \\
  &\stackrel{(3)}{=} n-\kappa(\mtx{A})^{-2} (n-\delta(C)) \\
  & = \kappa(\vct A)^{-2} \cdot \sdim(C) + \left(1-\kappa(\vct A)^{-2}\right)\cdot n,
\end{align*}
where for (1) we used~\eqref{eq:complement} and Lemma~\ref{lem:A^(-1)(D^*)}, for (2) we used~Theorem~\ref{thm:main-cond-bound}, and for (3) we used~\eqref{eq:complement} again.
\end{proof}

We conclude this section by proving Theorem~\ref{thm:B}, which we restate for convenience. 

\begin{theorem}
Let $C\subseteq \R^n$ be a closed convex cone and $\mtx{A}\in \R^{p\times n}$ have full rank. Let $\eta\in (0,1)$ and assume that $m\geq \delta(C)+2\sqrt{\log(2/\eta)m}$. Then
\begin{equation*}
  \delta(\mtx{A}C)\leq \overline{\mathcal{R}}^2_{C,m}(\mtx{A}) \cdot \delta(C)+(n-m)\eta.
\end{equation*}
For the matrix condition number,
\begin{equation}\label{eq:pqc-bound}
  \delta(\mtx{A}C) \leq \overline{\kappa}_{m}^2(\mtx{A}) \cdot \delta(C)+(n-m)\eta.
\end{equation}
\end{theorem}

\begin{proof}
The upper bound follows from
\begin{align*}
  \delta(\mtx{A}C) &\leq \Expect_{\mtx{Q}}[\delta(\mtx{P}_m\mtx{QA}C)]+(n-m)\eta\\
  & \leq \Expect_{\mtx{Q}}\Big[\Ren_C(\mtx{P}_m\mtx{QA})^2\Big]\delta(C)+(n-m)\eta,
\end{align*}
where we used Theorem~\ref{thm:main-cond-bound} for the second inequality. The upper bound in terms for the matrix condition number follows as in the proof of Theorem~\ref{thm:main-cond-bound}.
\end{proof}

\section{Applications}\label{sec:applications}
In this section we apply the results derived for convex cones to the setting of convex regularizers. To give this application some context, we briefly review some of the theory.

\subsection{Convex regularization, subdifferentials and the descent cone}\label{sec:conv-anal}
In practical applications the cones of interest often arise as cones generated by the subgradient of a proper convex function $f\colon \R^n\to \R\cup\{\infty\}$. 
The exact form of the general convex regularization problem is
\begin{equation}\label{eq:convreg}
 \minimize f(\vct{x}) \subjto \mtx{\Omega}\vct{x}=\vct{b},
\end{equation}
while the noisy form is
\begin{equation}\label{eq:convreg-noise}
\minimize f(\vct{x}) \subjto \norm{\mtx{\Omega}\vct{x}-\vct{b}}_2\leq \veps.
\end{equation}
Interchanging the role of the function $f$ and the residual, we get the \textit{generalized LASSO}
\begin{equation}\label{eq:convreg-lasso}
 \minimize \norm{\mtx{\Omega}\vct{x}-\vct{b}}_2 \subjto f(\vct{x})\leq \tau.
\end{equation}
Finally, we have the Lagrangian form,
\begin{equation}\label{eq:convreg-lagrange}
 \minimize \norm{\mtx{\Omega}\vct{x}-\vct{b}}_2^2 + \lambda f(\vct{x}).
\end{equation}

These last three problems are, in fact, equivalent (see~\cite[Chapter 3]{FR:13} for a concise derivation in the case $f(\vct{x})=\norm{\vct{x}}_1$).  The practical problem consists in effectively finding the parameters involved.

The first-order optimality condition states that $\hat{\vct{x}}$ is a unique solution of~\eqref{eq:convreg} if and only if
\begin{equation}\label{eq:opt}
 \exists \vct{y}\neq \zerovct \colon \mtx{\Omega}^{T}\vct{y} \in \partial f(\hat{\vct{x}}),
\end{equation}
where $\partial f(\hat{\vct{x}})$ denotes the subdifferential of $f$ at $\hat{\vct{x}}$, i.e., the set
\begin{equation*}
 \partial f(\hat{\vct{x}}) = \{\vct{z}\in \R^n \mid f(\hat{\vct{x}}+\vct{z})\geq f(\hat{\vct{x}})+\ip{\vct{z}}{\vct{x}} \}.
\end{equation*}
If $f$ is differentiable at $\hat{\vct{x}}$, then of course the subdifferential contains only the gradient of $f$ at $\hat{\vct{x}}$, and the vector $\vct{y}$ in~\eqref{eq:opt} consists of the Lagrange multipliers.

\begin{example}
If $f$ is a norm, with dual norm $f^\circ$, then the subdifferential of $f$ at $\hat{\vct{x}}$ is
\begin{equation*}
 \partial f(\hat{\vct{x}}) = \begin{cases} \{ \vct{z}\in \R^n \mid f^\circ(\vct{z})=1, \ip{\vct{z}}{\hat{\vct{x}}}=f(\hat{\vct{x}})\} & \hat{\vct{x}}\neq \zerovct\\
                              \{\vct{z}\in \R^n \mid f^\circ(\vct{z})\leq 1\} & \hat{\vct{x}}=\zerovct.
                             \end{cases}
\end{equation*}
\end{example}

\begin{example}
For the $\ell_1$-norm at an $s$-sparse vector $\hat{\vct{x}}$,
 \begin{equation*}
 \partial \lone{\hat{\vct{x}}} =  \{ \vct{z}\in \R^n \mid \norm{\vct{z}}_\infty=1, \ip{\vct{z}}{\hat{\vct{x}}}=\lone{\hat{\vct{x}}}\},
\end{equation*}
or more explicitly,
\begin{equation}\label{eq:l1subdif}
 \partial \lone{\hat{\vct{x}}} =  \{ \vct{z}\in \R^n \mid z_i=\mathrm{sign } \ (\hat{x}_i) \text{ if }\hat{x}_i\neq 0, \ z_j\in [-1,1] \text{ if } \hat{x}_j=0 \}.
\end{equation}
\end{example}

The descent cone of $f$ at $\hat{\vct{x}}$ is defined as
\begin{equation*}
 \Desc(f,\hat{\vct{x}}) = \bigcup_{\tau>0} \left\{\vct{y}\in \R^n\mid f(\hat{\vct{x}}+\tau \vct{y})\leq f(\hat{\vct{x}})\right\}.
\end{equation*}
The convex cone generated by the subdifferential of $f$ at $\hat{\vct{x}}$ is the closure of the polar cone of $\Desc(f,\hat{\vct{x}})$,
\begin{equation}\label{eq:desc-subd-dual}
 \mathrm{cone}\left(\partial f(\hat{\vct{x}})\right) = \overline{\Desc(f,\hat{\vct{x}})^{\circ}},
\end{equation}
Condition~\eqref{eq:opt} is therefore equivalent to
\begin{equation*}
 \ker \mtx\Omega \cap \Desc(f,\hat{\vct{x}}) = \{\zerovct\},
\end{equation*} 
namely, that the kernel of $\mtx{\Omega}$ does not intersect the descent cone nontrivially.

 An important class of regularizers are of the form $f(\vct{x}):=g(\mtx{A}\vct{x})+h(\mtx{B}\vct{x})$, with $\mtx{A}$ and $\mtx{B}$ linear maps. It follows from~\cite[Theorems 23.8, 23.9]{Rock} that the subdifferential is
\begin{equation*}
 \partial f(\vct{x}) = \mtx{A}^T\partial g(\mtx{A}\vct{x})+\mtx{B}^T\partial h(\mtx{B}\vct{x}).
\end{equation*}

\begin{example}
In the $\ell_1$-analysis, or cosparse, model, one considers regularizers of the form $\|\mtx{D}\vct{x}\|_1$, with $\mtx{D}\in \R^{p\times n}$ with typically $p\geq n$. The interest is on vectors for which $\mtx{D}\vct{x}_0$ is $s$-sparse. If $\mtx{D}$ has full rank and $\vct{x}_0\neq \zerovct$, then $s\geq p-n+1$, as otherwise $\mtx{D}$ would have a $n\times n$ minor that maps $\vct{x}_0$ to $\vct{0}$. The focus in this model has traditionally been on the {\em cosupport}, i.e., the location of the entries of $\mtx{D}\vct{x}_0$ that vanish. A typical example would be a shift invariant wavelet transform. The subdifferential of $\|\mtx{D}\cdot\|_1$ is given by $\mtx{D}^T\partial\|\mtx{D}\vct{x}_0\|_1$. For invertible $\mtx{D}$, combining~\eqref{eq:desc-subd-dual} with Lemma~\ref{lem:A^(-1)(D^*)} we get,
\begin{equation}\label{eq:descent-dual}
  \Desc(\|\mtx{D}\cdot\|_1,\vct{x}_0) = \mtx{D}^{-1} \Desc(\|\cdot\|_1,\mtx{D}\vct{x}_0).
\end{equation}
When working with the subdifferential cone rather than the descent cone, we don't need the invertibility requirement.
\end{example}

\begin{example}[Finite differences]
 Let $\vct{x}\in\R^n$ and let
 \begin{equation}\label{eq:fin-div}
  \mtx{D} = \begin{pmatrix} 
             -1 & 1 & 0 & \cdots & 0 & 0\\
             0 & -1 & 1 & \cdots & 0 & 0\\
             0 & 0 & -1 & \cdots & 0 & 0\\
             \vdots & \vdots & \vdots & \ddots & \vdots & \vdots \\
             0 & 0 & 0 & \cdots & -1 & 1\\
            0 & 0 & 0 & \cdots & 0 & -1
            \end{pmatrix}
 \end{equation}
be the discrete finite difference matrix. Thus 
\begin{equation*}
 \mtx{D}\vct{x}=(x_2-x_1,x_3-x_2,\dots,x_d-x_{d-1},-x_d)^{T}.
\end{equation*}
Define $g(\vct{x}) := f(\mtx{D}\vct{x})$. Then
for a fixed $\hat{\vct{x}}$, the subdifferential is given by
\begin{equation*}
 \partial g(\hat{\vct{x}}) = \mtx{D}^T \partial f(\mtx{D}\hat{\vct{x}}).
\end{equation*}
In the special case where $f$ is the $\ell_1$-norm and $\mtx{D}\hat{\vct{x}}$ is $s$-sparse with support $I\subset [n]$,
\begin{equation*}
 \partial g(\hat{\vct{x}}) = \{ \mtx{D}^T\vct{z} \mid \norm{\vct{z}}_\infty = 1, \ip{\vct{z}}{\mtx{D}\hat{\vct{x}}}=\norm{\mtx{D}\hat{\vct{x}}}_1\}.
\end{equation*}
One can think of such a vector $\hat{\vct{x}}$ as a signal with sparse gradient.
\end{example}

\begin{example}
 (Weighted $\ell_1$ norm). Let $\vct{\omega}\in \R^n$ be a vector of weights and define the weighted $\ell_1$-norm
 \begin{equation*}
  \norm{\vct{x}}_{\vct{\omega},1} = \sum_{j=1}^n \omega_j |x_j|.
 \end{equation*}
By extension from the $\ell_1$ example, we have
\begin{align*}
 \partial \norm{\hat{\vct{x}}}_{\vct{\omega},1} &=  \{ \vct{z}\in \R^n \mid z_i=\omega_i\ \mathrm{sign } \ (\hat{x}_i) \text{ if }\hat{x}_i\neq 0, \ z_j\in [-\omega_j,\omega_j] \text{ if } \hat{x}_j=0 \} \\
 &= \diag(\vct{\omega}) \ \partial \lone{\hat{\vct{x}}}.
\end{align*}
This example becomes interesting when considering {\em weighted} $s$-sparse vectors, that is, vectors such that
\begin{equation*}
 \norm{\vct{x}}_{\vct{\omega},0} = \sum_{x_j\neq 0} \omega_j^2 = s.
\end{equation*}
\end{example}

The use of composite regularizers to recover simultaneously structured models was studied in~\cite{oymak2015simultaneously}.

\subsection{Performance bounds in convex regularization}\label{sec:performance-bounds}
As mentioned in the introduction, computing the statistical dimension of convex regularizers is in general a difficult problem, with only few cases allowing for closed-form expressions. 
Using the condition bounds for the statistical dimension of linear images of convex cones, and translating these to the setting of convex regularizers, we get the corresponding statements in Corollary~\ref{cor:fromA}, which we restate here.

\begin{corollary}
Let $f(\vct{x}) = g(\vct{D}\vct{x})$, where $g$ is a proper convex function and let $\mtx{D}\in \R^{n\times n}$ be non-singular. Then
\begin{equation*}
  \delta(f,\vct{x}_0) \leq \Ren_{\Desc(g,\mtx{D}\vct{x}_0)}\left(\mtx{D}^{-1}\right) \cdot \delta(g,\mtx{D}\vct{x}_0).
\end{equation*}
In particular, 
\begin{equation*}
  \frac{\delta(g,\mtx{D}\vct{x}_0)}{\kappa(\mtx{D})^2}\leq \delta(f,\vct{x}_0) \leq \kappa(\mtx{D})^2 \cdot \delta(g,\mtx{D}\vct{x}_0).
\end{equation*}
\end{corollary}

\begin{proof}
Let $C=\Desc(g,\mtx{D}\vct{x}_0)$. Then from~\eqref{eq:descent-dual} we get that
\begin{equation*}
  \delta(f,\vct{x}_0) = \delta(\mtx{D}^{-1}C).
\end{equation*}
The claims then follows from Theorem~\ref{thm:main-cond-bound} and Proposition~\ref{prop:improved-cond}, noting that $\kappa(\mtx{D}^{-1})=\kappa(\mtx{D})$. 
\end{proof}

For convenience, we also recall the statement of Proposition~\ref{prop:l1analysis}. 

\begin{proposition}
Let $\mtx{D}\in \R^{p\times n}$, $p\geq n$, be such that all $n\times n$ minors of $\mtx{D}$ have full rank, and $\mtx{\Omega}\in \R^{m\times n}$ with $m\leq n$. Consider the problem
\begin{equation}\label{eq:l1min2}
  \minimize \norm{\mtx{D}\vct{x}}_1 \quad \subjto \mtx{\Omega}\vct{x}=\vct{b}.
\end{equation}
Let $\vct{x}_0\neq \zerovct$ be such that $\mtx{\Omega}\vct{x}_0=\vct{b}$, and such that $\vct{y}_0=\mtx{D}\vct{x}_0$ is $s$-sparse with support $I\subset [p]$. Let $\mtx{C}\in \R^{n\times p-s+1}$ be a matrix whose first $p-s$ columns consist of the columns of $\mtx{D}^T$ that are indexed by $I^{c}$, and the last column is $\vct{c}_{p-s+1}=\frac{1}{\sqrt{s}}\sum_{j\in I}\mathrm{sign}((\vct{y}_0)_j)\vct{d}_j$, where the vectors $\vct{d}_j$ denote the columns of $\mtx{D}^T$. 
Then
\begin{equation*}
  \delta(\norm{\mtx{D}\cdot}_1,\vct{x}_0) \leq \kappa(\mtx{C})^{-2} \cdot \delta(\norm{\cdot}_1,\mtx{D}\vct{x}_0) + \left(1-(p/n)\kappa(\mtx{C})^{-2}\right) \cdot n
\end{equation*}
In particular, given $\eta\in (0,1)$, Problem~\eqref{eq:l1min2} with Gaussian measurement matrix succeeds with probability $1-\eta$ if
\begin{equation*}
  m\geq \kappa(\vct{C})^{-2}\cdot \delta(\norm{\cdot}_1,\mtx{D}\vct{x}_0)+\left(1-(p/n)\kappa(\mtx{C})^{-2}\right)\cdot n+a_{\eta} \sqrt{n}, 
\end{equation*}
\end{proposition}

\begin{proof}
Set $f(\vct{x}) = \lone{\mtx{D}\vct{x}}$ with $\mtx{D}\in \R^{p\times n}$
and $p\geq n$.  Let $\vct{x}_0$ be given such that $\vct{y}_0=\mtx{D}\vct{x}_0$ is $s$-sparse with support $I$. Assuming that all the $n\times n$ minors of $\mtx{D}$ has rank $n$ and $\vct{x}_0\neq \zerovct$, $\vct{y}_0$ has at most $n-1$ zero entries, and the support therefore satisfies $s\geq p-n+1$. 
As shown in Section~\ref{sec:applications}, the descent cone $\Desc(f,\vct{x}_0)$ is polar to the subdifferential cone $\cone(\partial f(\vct{x}_0))$. Moreover, the statistical dimension satisfies $\delta(C)+\delta(C^{\circ})=n$, so that 
\begin{equation*}
 \delta(\Desc(f,\vct{x}_0)) = n-\delta(\cone(\partial f(\vct{x}_0))) = n-\delta(\mtx{D}^T\cone(\partial \|\vct{y}_0\|_1)).
\end{equation*}
The subdifferential of the $1$-norm is given by (see~\eqref{eq:l1subdif})
\begin{equation*}
\partial \lone{\vct{y}_0} =  \{ \vct{z}\in \R^p \mid z_i=\mathrm{sign } \ ((y_0)_i) \text{ if } i\in I, \ z_j\in [-1,1] \text{ if } j\not\in I \},
\end{equation*}
and we denote by $C:=\cone (\partial \lone{\vct{y}_0})$ the cone generated by this subdifferential.

\begin{figure}[htb!]
\centering
\includegraphics[width=0.5\textwidth]{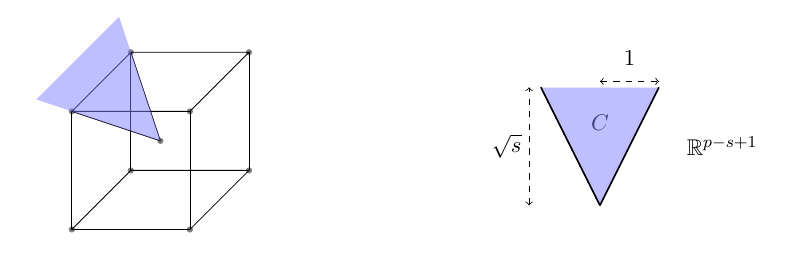}
\caption{Cone spanned by $(p-s)$-face of $d$-dimensional hypercube}\label{fig:hypercube}
\end{figure}

It follows that the cone generated by this subdifferential is contained in a subspace $L$ of dimension $\dim L=p-s+1\leq n$. An orthonormal basis of this subspace is given by the columns of a matrix $\mtx{B}=[\vct{b}_1,\dots,\vct{b}_{p-s+1}]$, where for $1\leq i\leq p-s$, the $\vct{b}_i$ are the unit vectors $\vct{e}_j$ for $j\in I^{c}$ and $\vct{b}_{p-s+1}=\frac{1}{\sqrt{s}}\sum_{j\in I} \mathrm{sign}((\vct{y}_0)_j) \vct{e}_j$. A moment's thought shows that $C=\mtx{B}\tilde{C}$, where $\tilde{C}\subset \R^{p-s+1}$ is the cone in $\R^{p-s+1}$ spanned by vectors of the form $\pm \vct{e}_i+\sqrt{s}\vct{e}_{p-s+1}$ for $1\leq i\leq n-p$ (see Figure~\ref{fig:hypercube}).
By the orthogonal invariance and the embedding invariance of the statistical dimension (see Properties (a) and (c) in Section~\ref{sec:stat-dim}), we get $\delta(C)=\delta(\tilde{C})$.
With this setup, we have
\begin{equation*}
 \mtx{D}^TC = \mtx{D}^T\mtx{B}\tilde{C} = \mtx{C}\tilde{C},
 \end{equation*}
 with the matrix $\mtx{C}:=\mtx{D}^{T}\mtx{B}\in \R^{n\times (p-s+1)}$ is
 then given as in the statement of the theorem. Applying the bounds from Theorem~\ref{thm:main-cond-bound} we thus get
 \begin{align*}
 \delta(\Desc(f,\vct{x}_0)) &= n-\delta(\mtx{D}^TC)\\
 & =  n-\delta(\mtx{C}\tilde{C})\\
 & \leq n-\kappa^{-2}(\mtx{C}) \delta(\tilde{C})\\
 &= n-\kappa^{-2}(\mtx{C}) \delta(C)\\
 & = \kappa(\mtx{C})^{-2} \cdot \delta(\norm{\cdot}_1,\mtx{D}\vct{x}_0) + \left(1-(p/n)\kappa(\mtx{C})^{-2}\right) \cdot n,
 \end{align*}
 as was to be shown. 
\end{proof}

\subsection{A note on the Stojnic method}
A popular method~\cite[Recipe 4.1]{edge}, going back to Stojnic~\cite{stojnic10} and generalized in~\cite{CRPW:12}, is to approximate the statistical dimension of the descent cone $\mathcal{D}(f,\vct{x}_0)$ by the expected value
\begin{equation}\label{eq:approx}
  \inf_{\tau\geq 0} \Expect[\dist^2(\vct{g},\tau\cdot \partial f(\vct{x}))].
\end{equation}
This approximation, however, does not work for all regularizers $f$ for two reasons: it my not be tight, and computing the quantity may not be feasible. 
In~\cite[Theorem 4.1]{edge}, the following error bound is derived:

\begin{equation}\label{eq:reg-bound}
  0 \leq \inf_{\tau\geq 0} \Expect[\dist^2(\vct{g},\tau\cdot \partial f(\vct{x}))] - \delta(f,\vct{x}_0) \leq \frac{2\sup\{\norm{\vct{s}}\colon \vct{s}\in \partial f(\vct{x})\}}{f(\vct{x}/\norm{\vct{x}})}.
\end{equation}

In~\cite{zhang2016precise}, the error~\eqref{eq:reg-bound} was analyzed in the case of TV minimization and it was shown to be bounded, so that the approximation is asymptotically tight. If $f(\vct{x})=\|\mtx{D}\vct{x}\|_1$ and assuming that $\vct{y}_0=\mtx{D}\vct{x}_0$ is $s$-sparse, we can express this bound in terms of the condition number of $\mtx{D}$. First note that the subdifferential of the $1$-norm is contained in the unit cube:
\begin{equation*}
  \partial \|\vct{y}_0\|_1 \subset \{\vct{z} \colon \|\vct{z}\|_\infty\leq 1\}.
\end{equation*}
Using the expression for the subdifferential of $g$ at $\vct{x}_0$, namely $\partial g(\vct{x}_0)=\mtx{D}^T\partial \|\vct{y}_0\|_1$, the error bound~\eqref{eq:reg-bound} translates to
\begin{equation*}
 \frac{2 \sup\{\|\vct{x}\|_2: \vct{x} \in \mtx{D}^T\partial \|\vct{y}_0\|_1\}}{\|\vct{y}_0\|_1/\|\vct{x}_0\|_2} \leq \frac{2}{\|\vct{y}_0\|_1/\|\vct{x}_0\|_2} \sup_{\|\vct{x}\|_{\infty}\leq 1} \|\mtx{D}^T\vct{x}\|.
\end{equation*}
Using the norm inequality $\|\vct{x}\|_2\leq \sqrt{n}\|\vct{x}\|_\infty$, we get the bound
\begin{equation*}
\sup_{\|\vct{x}\|_{\infty}\leq 1} \|\mtx{D}^T\vct{x}\|_2 \leq \sqrt{n}\sup_{\|\vct{x}\|_2 \leq 1} \|\mtx{D}^Tx\|_2 = \sqrt{n} \|\mtx{D}\|_2.
\end{equation*}
On the other hand, by the norm inequality
$\|\vct{x}\|_2\leq \|\vct{x}\|_1$ we have that
\begin{equation*}
  \frac{\|\vct{y}_0\|_1}{\|\vct{x}_0\|_2} =  \frac{\|\mtx{D} \vct{x}_0\|_1}{\|\vct{x}_0\|_2} \geq \frac{\|\vct{D} \vct{x}_0\|_2}{\|\vct{x}_0\|_2}\geq \sigma(\mtx{D}).
\end{equation*}
We therefore get the condition bound
\begin{equation*}
0 \leq \inf_{\tau\geq 0} \Expect[\dist^2(\vct{g},\tau\cdot \partial f(\vct{x}))] - \delta(f,\vct{x}_0) \leq \sqrt{n}\kappa(\mtx{D}).
\end{equation*}
From this we see that we can guarantee good bounds on the relative statistical dimension $\delta(f,\vct{x}_0)/n$ if the condition number of $\mtx{D}$ is small.
The bound can actually be improved when considering that we only need to maximize and minimize over certain subspaces in the definition of the singular values.

While this bound is not sharp (the derivation makes use of norm inequalities), it is enlightening as it gives sufficient conditions for the applicability of Bound~\eqref{eq:reg-bound} in terms of the condition number of $\mtx{A}$. It remains to be seen whether randomized preconditioning can be incorporated into this bound, and therefore whether this approach can lead to bounds that would rival those derived in~\cite{zhang2016precise}. 

\appendices

\section{The biconic feasibility problem - proofs}\label{sec:appendix-convex}

In this appendix we 
provide the proofs for Section~\ref{sec:gen-feas-prob}.
Recall that for $C\subseteq\IR^n$, $D\subseteq\IR^m$ closed convex cones, the biconic feasibility problem is given by
\\\def\tmpX{3mm}
\begin{minipage}{0.46\textwidth}
\begin{align}
   \exists \vct x & \in C\setminus\{\vct0\} \quad\text{s.t.} \hspace{\tmpX} \vct{Ax} \in D^\polar ,
\tag{P}
\label{eq:(P1)}
\end{align}
\end{minipage}
\rule{0.07\textwidth}{0mm}
\begin{minipage}{0.46\textwidth}
\begin{align}
   \exists \vct y & \in D\setminus\{\vct0\} \quad\text{s.t.} \hspace{\tmpX} -\vct A^T\vct y\in C^\polar ,
\tag{D}
\label{eq:(D1)}
\end{align}
\end{minipage}
\\[2mm] and the sets of primal feasible and dual feasible instances can be characterized by
\begin{align*}
   \P(C,D) & = \big\{\vct A\in\IR^{m\times n}\mid C\cap \big(\vct A^T D\big)^\polar\neq\{\vct0\}\big\} \\
   &= \{\vct A\in\IR^{m\times n}\mid \sres{A}{C}{D}=0 \} ,
\\ \D(C,D) & = \{\vct A\in\IR^{m\times n}\mid D\cap (-\vct A C)^\polar\neq\{\vct0\}\} \\
&=  \{\vct A\in\IR^{m\times n}\mid \srestm{A}{D}{C}=0 \} ,
\end{align*}
respectively, cf.~\eqref{eq:P(C,D)}/\eqref{eq:D(C,D)}. The proof of Proposition~\ref{prop:primaldual} uses the following generalization of Farkas' Lemma.

\begin{lemma}\label{lem:farkas-gen}
Let $C,\tilde C\subseteq\IR^n$ be closed convex cones with $\inter(C)\neq\emptyset$. Then
\begin{equation}\label{eq:Farkas-gen}
  \inter(C)\cap \tilde C=\emptyset \iff C^\polar\cap (-\tilde C^\polar)\neq\{\vct0\} .
\end{equation}
\end{lemma}

\begin{proof}
If $\inter(C)\cap \tilde C=\emptyset$, then there exists a separating hyperplane $H=\vct v^\bot$, $\vct v\neq\vct0$, so that $\langle \vct v,\vct x\rangle \leq 0$ for all $\vct x\in C$ and $\langle \vct v,\vct y\rangle \geq 0$ for all $\vct y\in \tilde C$. 
But this means $\vct v\in C^\polar\cap (-\tilde C^\polar)$.
On the other hand, if $\vct x\in \inter(C)\cap \tilde C$ then only in the case $C=\IR^n$, for which the claim is trivial, can $\vct x=\vct0$. If $\vct x\neq\vct0$, then~$C^\polar\setminus\{\vct0\}$ lies in the open half-space~$\{\vct v\mid \langle \vct v,\vct x\rangle<0\}$ and $-\tilde C^\polar$ lies in the closed half-space $\{\vct v\mid \langle \vct v,\vct x\rangle\geq0\}$, and thus $C^\polar\cap(-\tilde C^\polar)=\{\vct0\}$.
\end{proof}

For the proof of the third claim in Proposition~\ref{prop:primaldual} we also need the following well-known convex geometric lemma; a proof can be found, for example, in~\cite[proof of Thm.~6.5.6]{SW:08}. We say that two cones $C,D\subseteq\IR^n$, with $\inter(C)\neq\emptyset$, \emph{touch} if $C\cap D\neq\{\vct0\}$ but $\inter(C)\cap D=\emptyset$.

\begin{lemma}\label{lem:touch}
Let $C,D\subseteq\IR^n$ closed convex cones with $\inter(C)\neq\emptyset$. If $\vct Q\in O(n)$ uniformly at random, then the randomly rotated cone $\vct QD$ almost surely does not touch~$C$.
\end{lemma}

\begin{proof}[Proof of Proposition~\ref{prop:primaldual}]
(1) The sets $\P(C,D)$ and $\D(C,D)$ are closed as they are preimages of the closed set $\{0\}$ under continuous functions,
c.f.~\eqref{eq:P(C,D)}/\eqref{eq:D(C,D)}. Indeed, for any $\vct{x}$, the function $\mtx{A}\mapsto \norm{\Proj_D(\mtx{A}\vct{x})}$ is continuous, and as a minimum of such functions over the compact set $C\cap S^{m-1}$, it
follows that $\sigma_{C\to D}(\mtx{A})$ is continuous. Hence, $\P(C,D)=\{\vct A\in\IR^{n\times m}\mid \sres{A}{C}{D}=0 \}$ is closed.
The same argument applies to $\D(C,D)$.

(2) For the claim about the union of the sets $\P(C,D)$ and $\D(C,D)$ we first consider the case $C\neq\IR^n$, so that $\vct 0\not\in\inter(C)$. Using the generalized Farkas' Lemma~\ref{lem:farkas-gen}, we obtain
\begin{align*}
   \vct A\not\in\P(C,D) & \iff C\cap \big(\vct A^T D\big)^\polar = \{\vct0\} \;\\
   &\Rightarrow\;  \inter(C)\cap \big(\vct A^T D\big)^\polar = \emptyset\\
   & \stackrel{\eqref{eq:Farkas-gen}}{\Longrightarrow} C^\polar\cap (-\vct A^T D)\neq\{\vct0\} \;\Rightarrow\; \vct A\in\D(C,D) .
\end{align*}
This shows $\P(C,D)\cup\D(C,D)=\IR^{n\times m}$. For $D\neq\IR^n$ the argument is the same.
For $C=\IR^n$ and $D=\IR^m$:
\begin{align*}
   \P(\IR^n,\IR^m) & = \big\{\vct A\in\IR^{m\times n}\mid \ker\vct A\neq\{\vct0\}\big\} \\
   &= \begin{cases} \{\text{rank deficient matrices}\} & \text{if } n\leq m \\ \IR^{m\times n} & \text{if } n>m , \end{cases}
\\ \D(\IR^n,\IR^m) & = \big\{\vct A\in\IR^{m\times n}\mid \ker\vct A^T\neq\{\vct0\}\big\} \\
&= \begin{cases} \IR^{m\times n} & \text{if } n<m \\ \{\text{rank deficient matrices}\} & \text{if } n\geq m . \end{cases}
\end{align*}
In particular,this shows $\P(\IR^n,\IR^n)\cup \D(\IR^n,\IR^n)=\{\text{rank deficient matrices}\}$.

(3) If $(C,D)=(\IR^n,\IR^m)$ then by the characterization above $\Sigma(\IR^n,\IR^m)$ consists of the rank deficient matrices, which is a nonempty set. If $(C,D)\neq(\IR^n,\IR^n)$, then the union of the closed sets $\P(C,D)$ and $\D(C,D)$ equals~$\IR^{m\times n}$, which is an irreducible topological space, so that their intersection $\Sigma(C,D)=\P(C,D)\cap\D(C,D)$ must be nonempty.

As for the claim about the Lebesgue measure of $\Sigma(C,D)$, we may use the symmetry between~\eqref{eq:(P1)} and~\eqref{eq:(D1)} to assume without loss of generality $m\leq n$. If $\vct A\in\IR^{m\times n}$ has full rank, then $\vct AC$ has nonempty interior and from Proposition~\ref{prop:nres=0,sres=0} and Farkas' Lemma,
\begin{align*}
   \sres{A}{C}{D}=0 & \iff C\cap (\vct A^TD)^\polar\neq\{\vct0\}\\
   & \iff \vct AC\cap D^\polar\neq\{\vct0\} \;\text{or}\; \ker\vct A\cap C\neq\{\vct0\} ,
\\ \srestm{A}{D}{C}=0 & \iff D\cap (-\vct AC)^\polar\neq\{\vct0\}\\
& \stackrel{\eqref{eq:Farkas-gen}}{\iff} D^\polar\cap \inter(\vct AC)=\emptyset .
\end{align*}
Note that if $\vct{Ax}=\vct0$ for some $\vct x\in \inter(C)$, then~$\vct A$, being a continuous surjection, maps an open neighborhood of~$\vct x$ to an open neighborhood of the origin, so that~$\vct AC=\IR^m$. Hence, $D\cap (-\vct AC)^\polar\neq\{\vct0\}$ implies $\ker\vct A\cap \inter(C)=\emptyset$, since otherwise $\vct AC=\IR^m$, i.e., $(\vct AC)^\polar=\{\vct0\}$.

If $\vct A\in\Sigma(C,D)$, i.e., $\sres{A}{C}{D}=\srestm{A}{D}{C}=0$, and if $\vct A$ has full rank, then $\vct AC\cap D^\polar\neq\{\vct0\}$ implies that~$D^\polar$ touches~$\vct AC$, while $\ker\vct A\cap C\neq\{\vct0\}$ implies that~$\ker \vct A$ touches~$C$. Hence, if $\vct A=\vct G$ Gaussian, then $\vct G$ has almost surely full rank, and Lemma~\ref{lem:touch} implies that both touching events have zero probability, so that almost surely $\vct G\not\in\Sigma(C,D)$.
\end{proof}

We next provide the proof for the characterization of the restricted singular values as distances to the primal and dual feasible sets. From now on we use again the short-hand notation $\P:=\P(C,D)$ and $\D:=\D(C,D)$.

\begin{proof}[Proof of Proposition~\ref{prop:sing-dist}]
By symmetry, it suffices to show that $\dist(\vct A,\P)=\sres{A}{C}{D}$. If $\vct A\in\P$ then $\dist(\vct A,\P)=0=\sres{A}{C}{D}$, so assume that $\vct A\not\in\P$. Let $\DA\in\IR^{m\times n}$ such that $\vct A+\DA\in\P$ and $\dist(\vct A,\P)=\|\DA\|$. 
Since $\vct A+\DA\in\P$, there exists
$\vct x_0\in C\cap S^{n-1}$ such that $\vct w_0:=(\vct A+\DA)\vct x_0\in 
D^\polar$. For all $\vct y\in D$
  \[ 0\geq \langle\vct w_0,\vct y\rangle = \langle (\vct A+\DA)\vct x_0,\vct y\rangle = \langle \vct{Ax}_0,\vct y\rangle - \langle -\DA\vct x_0,\vct y\rangle . \]
If $\vct y_0\in B^m\cap D$ is such that $\|\Pi_D(\vct{Ax}_0)\|=\langle \vct{Ax}_0,\vct y_0\rangle$, then
\begin{align*}
   \dist(\vct A,\P) & = \|\DA\| \geq \|\DA\vct x_0\| 
   \geq \|\Pi_D(-\DA\vct x_0)\| \\
   &= \max_{\vct y\in B^m\cap D}\langle -\DA\vct x_0,\vct y\rangle
\\ & \geq \langle -\DA\vct x_0,\vct y_0\rangle \geq \langle \vct{Ax}_0,\vct y_0\rangle = \|\Pi_D(\vct{Ax}_0)\|\\
& \geq \min_{\vct x\in C\cap S^{n-1}} \|\Pi_D(\vct{Ax})\| = \sres{A}{C}{D} .
\end{align*}

For the reverse inequality $\dist(\vct A,\P)\leq \sres{A}{C}{D}$ we need to construct a perturbation~$\DA$ such that $\vct A+\DA\in\P$ and $\|\DA\|\leq\sres{A}{C}{D}$. Let $\vct x_0\in C\cap S^{n-1}$ and $\vct y_0\in D\cap B^m$ such that
  \[ \sres{A}{C}{D} = \min_{\vct x\in C\cap S^{n-1}} \max_{\vct y\in D\cap B^m} \langle \vct{Ax},\vct y\rangle = \langle \vct{Ax}_0,\vct y_0\rangle . \]
Since $\vct A\not\in\P$ we have $\sres{A}{C}{D}>0$, which implies $\|\vct y_0\|=1$, i.e., $\vct y_0\in D\cap S^{m-1}$. We define
  \[ \DA := -\vct y_0\vct y_0^T\vct A . \]
Note that
  \[ \|\DA\| = \|\vct A^T\vct y_0\| \leq \langle \vct A^T\vct y_0,\vct x_0\rangle = \sres{A}{C}{D} . \]
Furthermore,
\begin{align*}
   (\vct A+\DA)\vct x_0 & = \vct{Ax}_0 - \vct y_0\vct y_0^T\vct{Ax}_0 \\
   &= \vct{Ax}_0 - \langle \vct{Ax}_0,\vct y_0\rangle\vct y_0 \\
   &= \vct{Ax}_0 - \Pi_D(\vct{Ax}_0) = \Pi_{D^\polar}(\vct{Ax}_0) .
\end{align*}
So $\vct x_0\in C\setminus\{\vct0\}$ and $(\vct A+\DA)\vct x_0\in D^\polar$, which shows that $\vct A+\DA\in\P$, and hence $\dist(\vct A,\P)\leq \|\DA\| \leq \sres{A}{C}{D}$.
\end{proof}

\section*{Acknowledgment}
The authors would like to thank Mike~McCoy and Joel Tropp for fruitful discussions on integral geometry, and in particular for suggesting the TQC Lemma, and Armin Eftekhari for helpful discussions on random projections. I would also like to thank the anonymous referees for valuable feedback and suggestions.

\ifCLASSOPTIONcaptionsoff
  \newpage
\fi



%

\def\cprime{$'$} \def\cprime{$'$}

%

\begin{IEEEbiographynophoto}{Dennis Amelunxen}
was Assistant Professor at the Mathematics Department of City University of Hong Kong. Dennis Amelunxen received his PhD from the University of Paderborn, Germany, in 2011. Before joining City University in 2014, he worked as a postdoctoral fellow at Cornell University, USA, and at The University of Manchester, UK.
\end{IEEEbiographynophoto}
\begin{IEEEbiographynophoto}{Martin Lotz}
is Associate Professor of Mathematics at the University of Warwick. 
Prior to joining Warwick, Martin Lotz was a Lecturer in Numerical Analysis at the University of Manchester and held research positions at the City University of Hong Kong, at the University of Oxford, and at the University of Edinburgh, supported by a Leverhulme Trust and Seggie Brown Fellowship. Martin Lotz received his undergraduate degree from the ETH Z\"urich, and his PhD at the University of Paderborn, with a thesis on Algebraic Complexity Theory.
\end{IEEEbiographynophoto}
\begin{IEEEbiographynophoto}{Jake Walvin}
completed his PhD in Mathematics at the University of Manchester in 2019.
\end{IEEEbiographynophoto}




\end{document}